\documentclass[12pt,reqno]{article}
\usepackage{amssymb,amscd,amsmath,amsthm,color}
\newtheorem{thm}{Theorem}[section]
\newtheorem{prop}[thm]{Proposition}
\newtheorem{lemma}[thm]{Lemma}
\newtheorem{cor}[thm]{Corollary}

\newtheorem{remark}[thm]{Remark}
\newtheorem{example}[thm]{Example}

\numberwithin{equation}{section}

\def\bR{\mathbb{R}}

\def\bN{\mathbb{N}}
\def\cH{\mathcal{H}}

\def\bP{\mathbb{P}}

\def\<{\langle}
\def\>{\rangle}
\def\dT{d_{\mathrm T}}
\def\eps{\varepsilon}

\def\cA{\mathcal{A}}

\def\cK{\mathcal{K}}

\newcommand{\UIN}[1]{\left|\!\left|\!\left|{#1}\right|\!\right|\!\right|}

\begin{document}
\baselineskip=15pt
\allowdisplaybreaks

\centerline{\LARGE Ando-Hiai type inequalities for}
\centerline{\LARGE multivariate operator means}

\bigskip
\bigskip
\centerline{\large
Fumio Hiai\footnote{{\it E-mail:} hiai.fumio@gmail.com},
Yuki Seo\footnote{{\it E-mail:} yukis@cc.osaka-kyoiku.ac.jp}
and Shuhei Wada\footnote{{\it E-mail:} wada@j.kisarazu.ac.jp}}

\medskip
\begin{center}
$^1$\,Tohoku University (Emeritus), \\
Hakusan 3-8-16-303, Abiko 270-1154, Japan
\end{center}

\begin{center}
$^2$\,Department of Mathematics Education, Osaka Kyoiku University, \\
Asahigaoka, Kashiwara, Osaka 582-8582, Japan
\end{center}

\begin{center}
$^3$\,Department of Information and Computer Engineering, \\
National Institute of Technology, \\
Kisarazu College, Kisarazu, Chiba 292-0041, Japan
\end{center}

\medskip
\begin{abstract}
We present several Ando-Hiai type inequalities for $n$-variable operator means for positive
invertible operators. Ando-Hiai's inequalities given here are not only of the original type
but also of the complementary type and of the reverse type involving the generalized
Kantorovich constant.


\end{abstract}

\section{Introduction}

The most studied ($2$-variable) operator mean is probably the (weighted) \emph{geometric mean}
$$
A\#_\alpha B:=A^{1/2}(A^{-1/2}BA^{-1/2})^\alpha A^{1/2}\qquad
\mbox{where}\quad0\le\alpha\le1
$$
for positive invertible operators $A$ and $B$. The geometric mean $\#=\#_{1/2}$ was first
introduced by Pusz and Woronowicz \cite{PW}, which was further developed into a general theory
of operator means by Kubo and Ando \cite{KA}. \emph{Ando-Hiai's inequality} \cite{AH} is an
operator inequality receiving much attention related to the geometric mean, which says that
for $0\le\alpha\le1$,
$$
A\#_\alpha B\le I\ \implies\ A^r\#_\alpha B^r\le I,\qquad r\ge1,
$$
or a bit more strongly,
\begin{align}\label{F-1.1}
A^r\#_\alpha B^r\le\|A\#_\alpha B\|_\infty^{r-1}(A\#_\alpha B),\qquad r\ge1,
\end{align}
where $\|\cdot\|_\infty$ is the operator norm. From this with use of the antisymmetric tensor
power technique, the so-called log-majorization for the geometric mean was obtained in
\cite{AH}. Among others, a generalization of Ando-Hiai's inequality in \cite{Wa} is worth
noting, where the class of operator means satisfying Ando-Hiai's inequality was characterized
by the power monotone increasing condition.

It was a long-standing open problem to extend the geometric mean to the case of more than two
variables of matrices or operators. The problem was finally settled by an iteration method in
\cite{ALM} and by a Riemannian geometry method in \cite{Mo,BH}. Since then, the latter
approach has extensively been advanced by many authors, e.g., \cite{LL0,LP,LL2,Pa}. Nowadays,
the multivariate geometric mean in the Riemannian geometric approach is often called the
\emph{Karcher mean} since it is determined as a solution of the so-called Karcher equation.
Another important multivariate operator mean under recent active consideration is the
\emph{power mean} developed in \cite{LP,LL2}.

The extension of Ando-Hiai's inequality to the Karcher mean was made by Yamazaki \cite{Ya1},
and its modification for the power mean was also shown in \cite{LY}. In the present paper we
aim to present different Ando-Hiai type inequalities  in the spirit of \cite{Wa} for as much
as general $n$-variable operator means. To do so, we first develop, in Section 2, a theory of
deformation for $n$-variable operator means, where we introduce the deformed mean $M_\sigma$
of an $n$-variable mean $M$ by a $2$-variable operator mean $\sigma$ based on a fixed point
method. This is considered in some sense as an extended version of the generalized operator
means by P\'alfia \cite{Pa}. In Section 3, we then consider Ando-Hiai's inequality for
$n$-variable means $M$ in such forms as
\begin{align}\label{F-1.2}
M(A_1^r,\dots,A_n^r)\le\|M(A_1,\dots,A_n)\|_\infty^{r-1}M(A_1,\dots,A_n)
\quad\mbox{for $r\ge1$},
\end{align}
and its complementary version
\begin{align}\label{F-1.3}
M(A_1^r,\dots,A_n^r)\ge\|M(A_1,\dots,A_n)\|_\infty^{r-1}M(A_1,\dots,A_n)
\quad\mbox{for $0<r\le1$}.
\end{align}
We prove a hereditary result that if $M$ satisfies \eqref{F-1.2} or \eqref{F-1.3} and
$\sigma$ is power monotone increasing, then the deformed mean $M_\sigma$ satisfies the same.
This derives Ando-Hiai's inequality for the power mean when $M$ is the weighted arithmetic
mean and $\sigma=\#_\alpha$. Moreover, we show that the above condition for $\sigma$ is
indeed necessary for $M_\sigma$ to satisfy Ando-Hiai's inequality when $M$ is, in particular,
the Karcher mean.

Next, in Section 4, we prove certain modifications of \eqref{F-1.2} and \eqref{F-1.3} (see
\eqref{F-4.1} and \eqref{F-4.2} for the precise forms) when $M$ is the deformed mean $M_\sigma$
of an arbitrary $n$-variable mean $M$ by an arbitrary $\sigma$ (except the left trivial mean).
When $M$ is specialized to a $2$-variable mean and $\sigma=\#_\alpha$, our inequalities here
include the generalized version in \cite{Wa} and the complementary version in \cite{Se}.
Furthermore, in Section 5, we obtain the reverse versions, involving the generalized
Kantorovich constant, of \eqref{F-1.2} (though restricted to the power mean) and of the
modification of \eqref{F-1.3}. We expect that the complementary and reverse versions of
Ando-Hiai's inequalities give a new perspective in the topic. Finally, in Section 6, the
optimality of the power $r\ge1$ or $0<r\le1$ in \eqref{F-1.2} and \eqref{F-1.3} is examined,
thus extending an optimality result in \cite{Wa2}.

Here it should be noted that in a recent paper \cite{Ya2} Yamazaki obtained two Ando-Hiai's
inequalities for $n$-variable generalized operator means in the sense of \cite{Pa}, which are
in the weaker formulation of the form $M(A_1,\dots,A_n)\le I$ $\implies$
$M(A_1^r,\dots,A_n^r)\le I$. One inequality in \cite{Ya2} will be incorporated in Theorem
\ref{T-3.6} and another is similar to Theorem \ref{T-4.1}.

\section{Multivariate means}

Throughout the paper, $\cH$ is a general Hilbert space assumed to be infinite-dimensional
unless otherwise stated, $B(\cH)$ is the algebra of all bounded linear operators on $\cH$,
$B(\cH)^+$ the set of positive operators in $B(\cH)$, and $\bP=\bP(\cH)$ the set of positive
invertible operators in $B(\cH)$. For $X\in\bP$, $\|X\|_\infty$ is the operator norm of $X$
and $\lambda_{\min}(X)$ is the minimum of the spectrum of $X$, that is,
$\lambda_{\min}(X)=\|X^{-1}\|_\infty^{-1}$. Moreover, $I$ denotes the identity operator, and
SOT means the strong operator topology on $B(\cH)$. The \emph{Thompson metric} $\dT$ on $\bP$
is defined by
$$
\dT(A,B):=\|\log A^{-1/2}BA^{-1/2}\|_\infty=\log\max\{M(A/B),M(B/A)\},
$$
where $M(A/B):=\inf\{\alpha>0:A\le\alpha B\}$. It is known \cite{Th} that $(\bP,\dT)$ is a
complete metric space and both topologies on $\bP$ induced by $\dT$ and $\|\cdot\|_\infty$
coincide.

The notion of $2$-variable operator means was introduced by Kubo-Ando \cite{KA} in an
axiomatic way as follows: A map $\sigma:B(\cH)^+\times B(\cH)^+\to B(\cH)^+$ is called an
\emph{operator mean} if it satisfies the following properties:
\begin{itemize}
\item[(i)] \emph{Monotonicity:} $A\le C$, $B\le D$ $\implies$ $A\sigma B\le C\sigma D$.
\item[(ii)] \emph{Transformer inequality:} $C(A\sigma B)C\le(CAC)\sigma(CBC)$ for every
$C\in B(\cH)^+$.
\item[(iii)] \emph{Downward continuity:} $A_k\searrow A$, $B_k\searrow B$ $\implies$
$A_k\sigma B_k\searrow A\sigma B$, where $A_k\searrow A$ means that $A_1\ge A_2\ge\cdots$ and
$A_k\to A$ in SOT.
\item[(iv)] \emph{Normalized condition:} $I\sigma I=I$.
\end{itemize}

The main theorem of \cite{KA} says that there is a one-to-one order-isomorphic and affine
correspondence $\sigma\leftrightarrow f$ between the operator means $\sigma$ and the
non-negative operator monotone functions $f$ on $(0,\infty)$ with $f(1)=1$ determined by
$$
f(x)I=I\sigma(xI),\qquad x>0,
$$
$$
A\sigma B=A^{-1/2}f(A^{-1/2}BA^{-1/2})A^{1/2},\qquad A,B\in\bP,
$$
which extends to general $A,B\in B(\cH)^+$ as
$A\sigma B=\lim_{\eps\searrow0}(A+\eps I)\sigma(B+\eps I)$ in SOT. The above operator
monotone function $f$ on $(0,\infty)$ corresponding to $\sigma$ is denoted by $f_\sigma$ and
called the \emph{representing function} of $\sigma$.

As a multivariate extension of operator means we call a map $M:\bP^n\to\bP$ an
\emph{$n$-variable operator mean} if it satisfies the following properties:
\begin{itemize}
\item[(I)] \emph{Monotonicity:} If $A_j,B_j\in\bP$ and $A_j\le B_j$ for $1\le j\le n$, then
$$
M(A_1,\dots,A_n)\le M(B_1,\dots,B_n).
$$
\item[(II)] \emph{Congruence invariance:} For every $A_1,\dots,A_n\in\bP$ and any invertible
$S\in B(\cH)$,
$$
S^*M(A_1,\dots,A_n)S=M(S^*A_1S,\dots,S^*A_nS).
$$
A special case of this is homogeneity: $M(tA_1,\dots,tA_n)=tM(A_1,\dots,A_n)$ for $t>0$.
\item[(III)] \emph{Monotone continuity:} Let $A_j,A_{jk}\in\bP$ for $1\le j\le n$ and
$k\in\bN$. If either $A_{jk}\searrow A_j$ as $k\to\infty$ for each $j$ or $A_{jk}\nearrow A_j$
as $k\to\infty$ for each $j$, then
$$
M(A_{1k},\dots,A_{nk})\ \longrightarrow\ M(A_1,\dots,A_n)\quad\mbox{in SOT}.
$$
\item[(IV)] \emph{Normalized condition:} $M(I,\dots,I)=I$.
\end{itemize}

Resemblances of (I)--(IV) for multivariate means to (i)--(iv) for operator means are
apparent, but there are also slight differences between those. For one thing, multivariate
means are maps restricted on $\bP^n$ while operator means are on $B(\cH)^+\times B(\cH)^+$.
For another, (II) is formally stronger than (ii) but we note \cite{KA} that congruence
invariance $S^*(A\sigma B)S=(S^*AS)\sigma(S^*BS)$ for invertible $S\in B(\cH)$ is automatic
for operator means $\sigma$. Moreover, we assume continuity both downward and upward in
(III) while only downward is assumed in (iii). Continuity from both directions seems natural
when we take care of transformation under $A\in\bP\mapsto A^{-1}\in\bP$ for means on $\bP^n$.
Here it is worth noting that any operator mean is also upward continuous when restricted to
$\bP\times\bP$.

Let $M$ be an $n$-variable operator mean and $\sigma$ be a $2$-variable operator mean as
stated above, and assume that $\sigma\ne\frak{l}$, where $\frak{l}$ is the left trivial mean
$A\frak{l}B=A$. For given $A_1,\dots,A_n\in\bP$ we consider the fixed point type equation
\begin{align}\label{F-2.1}
X=M(X\sigma A_1,\dots,X\sigma A_n)\qquad\mbox{for}\quad X\in\bP,
\end{align}
which is, due to (II), equivalent to
$$
I=M\bigl(I\sigma(X^{-1/2}A_1X^{-1/2}),\dots,I\sigma(X^{-1/2}A_nX^{-1/2})\bigr),
$$
that is,
\begin{align}\label{F-2.2}
I=M\bigl(f_\sigma(X^{-1/2}A_1X^{-1/2}),\dots,f_\sigma(X^{-1/2}A_nX^{-1/2})\bigr),
\end{align}
where $f_\sigma$ is the representing function of $\sigma$.

We show the next theorem concerning the solution to equation \eqref{F-2.1} or \eqref{F-2.2},
which gives a theoretical basis for our discussions below. In fact, part (3) of the theorem
will repeatedly be used in Sections 3--5.

\begin{thm}\label{T-2.1}
\begin{itemize}
\item[$(${\rm 1}$)$] For every $A_1,\dots,A_n\in\bP$ there exists a unique $X_0\in\bP$ which satisfies
\eqref{F-2.1}.
\item[$(${\rm 2}$)$] Write $M_\sigma(A_1,\dots,A_n)$ for the unique solution $X_0$ to \eqref{F-2.1} given
in $(${\rm 1}$)$. Then $M_\sigma:\bP^n\to\bP$ is an $n$-variable mean satisfying $(${\rm I}$)$--$(${\rm IV}$)$ again.
\item[$(${\rm 3}$)$] If $Y\in\bP$ and $Y\le M(Y\sigma A_1,\dots,Y\sigma A_n)$, then
$Y\le M_\sigma(A_1,\dots,A_n)$. If $Y\in\bP$ and $Y\ge M(Y\sigma A_1,\dots,Y\sigma A_n)$,
then $Y\ge M_\sigma(A_1,\dots,A_n)$.
\end{itemize}
\end{thm}

We call $M_\sigma$ given in (2) above the \emph{deformed mean} from $M$ by $\sigma$. To prove
the theorem, we first give two lemmas.

\begin{lemma}\label{L-2.2}
For every $A_j,B_j\in\bP$ for $1\le j\le n$,
\begin{align*}
\dT(M(A_1,\dots,A_n),M(B_1,\dots,B_n))\le\max_{1\le j\le k}\dT(A_j,B_j).
\end{align*}
\end{lemma}

\begin{proof}
The proof is standard from (I) and (II) while we give it for completeness. Let $\alpha$ be
the maximum in the right-hand side. Since $e^{-\alpha}A_j\le B_j\le e^\alpha A_j$ for
$1\le j\le n$, we have
$$
e^{-\alpha}M(A_1,\dots,A_n)=M(e^{-\alpha}A_1,\dots,e^{-\alpha}A_n)
\le M(B_1,\dots,B_n),
$$
and similarly $M(B_1,\dots,B_n)\le e^\alpha M(A_1,\dots,A_n)$. Hence the asserted inequality
follows.
\end{proof}

\begin{lemma}\label{L-2.3}
If $A,X,Y\in\bP$ and $X\ne Y$, then $\dT(X\sigma A,Y\sigma A)<\dT(X,Y)$.
\end{lemma}

\begin{proof}
First, we see that if $X,A,B\in\bP$ and $A<B$, then $X\sigma A<X\sigma B$. Here we write
$A<B$ to mean that $B-A\in\bP$. By assumption $\sigma\ne\frak{l}$, i.e.,
$f_\sigma\not\equiv1$, $f_\sigma$ is strictly increasing on $(0,\infty)$. 
Since $A<B$ and hence $X^{-1/2}AX^{-1/2}<X^{-1/2}BX^{-1/2}$, one has
$X^{-1/2}AX^{-1/2}+\eps I\le X^{-1/2}BX^{-1/2}$ for some $\eps>0$. Choose a $\delta\in(0,1)$
such that
$\delta I\le X^{-1/2}AX^{-1/2}\le\delta^{-1}I$. Since
\begin{align*}
&f_\sigma(X^{-1/2}AX^{-1/2}+\eps I)-f_\sigma(X^{-1/2}AX^{-1/2}) \\
&\qquad\ge\bigl(\min\{f_\sigma(t+\eps)-f_\sigma(t):\delta\le t\le\delta^{-1}\}\bigr)I,
\end{align*}
one has
\begin{align*}
f_\sigma(X^{-1/2}AX^{-1/2})<f_\sigma(X^{-1/2}AX^{-1/2}+\eps I)
\le f_\sigma(X^{-1/2}BX^{-1/2}),
\end{align*}
which implies that $X\sigma A<X\sigma B$.

Now, let $\alpha:=\dT(X,Y)>0$ (thanks to $X\ne Y$). Since $e^{-\alpha}A<A<e^\alpha A$, the
above shown fact gives
\begin{align*}
&Y\sigma A\le(e^\alpha X)\sigma A<(e^\alpha X)\sigma(e^\alpha A)=e^\alpha(X\sigma A), \\
&Y\sigma A\ge(e^{-\alpha}X)\sigma A>(e^{-\alpha}X)\sigma(e^{-\alpha}A)=e^{-\alpha}(X\sigma A).
\end{align*}
Therefore, $e^{-\beta}(X\sigma A)\le Y\sigma A\le e^\beta(X\sigma A)$ for some
$\beta\in(0,\alpha)$, which implies that $\dT(X\sigma A,Y\sigma A)\le\beta<\alpha$.
\end{proof}

\noindent
{\it Proof of Theorem \ref{T-2.1}}.\enspace
(1)\enspace
Choose a $\delta\in(0,1)$ such that $A_1,\dots,A_n\in\Sigma_\delta$, where
$\Sigma_\delta:=\{X\in\bP:\delta I\le X\le\delta^{-1}I\}$. Define the map $F:\bP\to\bP$ by
$$
F(X):=M(X\sigma A_1,\dots,X\sigma A_n),\qquad X\in\bP.
$$
It is immediate to see from (I), (II) and (IV) that $F$ maps $\Sigma_\delta$ into itself and
$F$ is monotone, i.e., if $X,Y\in\bP$ and $X\le Y$ then $F(X)\le F(Y)$. Let $Y_0:=\delta^{-1}I$.
Since $F(Y_0)\in\Sigma_\delta$, one has $Y_0\ge F(Y_0)$ so that
$Y_0\ge F(Y_0)\ge F^2(Y_0)\ge\cdots\ge\delta I$. Therefore, $F^k(Y_0)\searrow X_0$ for some
$X_0\in\bP$. Since $F^k(Y_0)\sigma A_j\searrow X_0\sigma A_j$ as $k\to\infty$, it follows
from (III) that
\begin{align*}
X_0&=\lim_{k\to\infty}F(F^k(Y_0))
=\lim_{k\to\infty}M(F^k(Y_0)\sigma A_1,\dots,F^k(Y_0)\sigma A_n) \\
&=M(X_0\sigma A_1,\dots,X_0\sigma A_n)\quad\mbox{in SOT}.
\end{align*}
Hence $X_0$ is a solution to \eqref{F-2.1}. To show the uniqueness of the solution, assume
that $X_0,X_1\in\bP$ satisfies \eqref{F-2.1} and $X_0\ne X_1$. By Lemmas \ref{L-2.2} and
\ref{L-2.3},
\begin{align*}
\dT(X_0,X_1)\le\max_{1\le j\le n}\dT(X_0\sigma_jA_j,X_1\sigma_jA_j)<\dT(X_0,X_1),
\end{align*}
a contradiction.

(3)\enspace
We prove this before (2). Let $X_0:=M_\sigma(A_1,\dots,A_n)$. Assume that $Y\in\bP$ and
$Y\le M(Y\sigma A_1,\dots,Y\sigma A_n)$. Then $Y\le F(Y)\le F^2(Y)\le\cdots$. Choose a
$\delta>0$ such that $Y,A_1,\dots,A_n\in\Sigma_\delta$. Since $Y\sigma A_j\le\delta^{-1}I$, one
has $F(Y)\le\delta^{-1}I$. Iterating this gives $F^k(Y)\le\delta^{-1}I$ for all $k$, hence
$F^k(Y)\nearrow Y_0$ for some $Y_0\in\bP$. As in the proof of (1), $F(Y_0)=Y_0$ due to (III)
and hence $Y\le Y_0=X_0$. The proof of the other assertion is similar.

(2)\enspace
Assume that $A_j\le B_j$ for $1\le j\le n$, and let $Y_0:=M_\sigma(B_1,\dots,B_n)$. Since
$Y_0=M(Y_0\sigma B_1,\dots,Y_0\sigma B_n)\ge M(Y_0\sigma A_1,\dots,Y_0\sigma A_n)$, one has
$Y_0\ge M_\sigma(A_1,\dots,A_n)$ by (3) proved above. Hence $M_\sigma$ satisfies (I). For any
$A_1,\dots,A_n\in\bP$ let $X_0:=M_\sigma(A_1,\dots,A_n)$. Then for any invertible
$S\in B(\cH)$ one has
\begin{align*}
S^*X_0S&=S^*M(X_0\sigma A_1,\dots,X_0\sigma A_n)S \\
&=M((S^*X_0S)\sigma(S^*A_1S),\dots,(S^*X_0S)\sigma(S^*A_nS)),
\end{align*}
showing that $S^*X_0S=M_\sigma(S^*A_1S,\dots,S^*A_nS)$. Hence $M_\sigma$ satisfies (II). Let
$A_j,A_{jk}\in\bP$ and assume that $A_{jk}\searrow A_j$ as $k\to\infty$ for $1\le j\le n$.
Let $X_k:=M_\sigma(A_{1k},\dots,A_{nk})$ so that
$X_k=M(X_k\sigma A_{1k},\dots,X_k\sigma A_{nk})$. By (I) for $M_\sigma$ we have $X_k\searrow$
as $k\to\infty$ and $X_k\ge M_\sigma(A_1,\dots,A_n)$. Hence $X_k\searrow X_0$ for some
$X_0\in\bP$. Since $X_k\sigma A_{jk}\searrow X_0\sigma A_j$ for $1\le j\le n$, it follows
from (III) for $M$ that $X_0=M(X_0\sigma A_1,\dots,X_0\sigma A_n)$, showing that
$X_0=M_\sigma(A_1,\dots,A_n)$. Hence $M_\sigma$ is downward continuous. The proof of upward
continuity is similar. Finally, since $M(I\sigma I,\dots,I\sigma I)=M(I,\dots,I)=I$,
$M_\sigma$ satisfies (IV).\qed

\bigskip
Let $M$ be an $n$-variable operator mean satisfying (I)--(IV). It is obvious that
$M_{\frak{r}}=M$, where $\frak{r}$ is the right trivial mean $A\frak{r}B=B$. The \emph{adjoint}
$M^*$ of $M$ is defined by
$$
M^*(A_1,\dots,A_n):=M(A_1^{-1},\dots,A_n^{-1})^{-1},\qquad A_j\in\bP.
$$
Then it is easy to verify that $M^*$ is again a mean satisfying (I)--(IV) and
$(M_\sigma)^*=(M^*)_{\sigma^*}$ holds for any operator mean $\sigma\ne\frak{l}$, where
$\sigma^*$ is the adjoint of $\sigma$, i.e., $A\sigma^*B:=(A^{-1}\sigma B^{-1})^{-1}$.

\begin{example}\label{E-2.4}\rm
Typical examples of multivariate means satisfying (I)--(IV) are in order. Let
$\omega=(w_1,\dots,w_n)$ be a probability vector, i.e., $w_j\ge0$ and $\sum_{j=1}^nw_j=1$.
\begin{itemize}
\item[(a)] The \emph{weighted arithmetic mean} $\cA_\omega$ and the \emph{weighted harmonic
mean} $\cH_\omega=(\cA_\omega)^*$ are
$$
\cA_\omega(A_1,\dots,A_n):=\sum_{j=1}^nw_jA_j,\qquad
\cH_\omega(A_1,\dots,A_n):=\Biggl(\sum_{j=1}^nw_jA_j^{-1}\Biggr)^{-1}
$$
for $A_j\in\bP$. It is obvious that $\cA_\omega$ and $\cH_\omega$ satisfy (I)--(IV).

\item[(b)] For each $\alpha\in[-1,1]\setminus\{0\}$ the (weighted) \emph{power mean}
$P_{\omega,\alpha}(A_1,\dots,A_n)$, introduced in \cite{LP,LL1,LL2}, is defined as the unique
solution to the equation
\begin{align*}
&X=\cA_\omega(X\#_\alpha A_1,\dots,X\#_\alpha A_n)\qquad\,\mbox{for $0<\alpha\le1$}, \\
&X=\cH_\omega(X\#_{-\alpha}A_1,\dots,X\#_{-\alpha}A_n)\quad\mbox{for $-1\le\alpha<0$},
\end{align*}
that is, for $0<\alpha\le1$,
$$
P_{\omega,\alpha}=(\cA_\omega)_{\#_\alpha},\qquad
P_{\omega,-\alpha}=(\cH_\omega)_{\#_\alpha}=(P_{\omega,\alpha})^*.
$$
By Theorem \ref{T-2.1}\,(2) and the above (a) we see that $P_{\omega,\alpha}$ satisfies
properties (I)--(IV), while those except (III) are included in \cite{LL2,Pa}.

\item[(c)] The \emph{Karcher mean} (the multivariate geometric mean) $G_\omega(A_1,\dots,A_n)$
is defined as the unique solution to the \emph{Karcher equation}
\begin{align*}
\sum_{j=1}^nw_j\log(X^{-1/2}A_jX^{-1/2})=0
\end{align*}
(see \cite{Mo,LP,LL2}). It is known \cite{LL2} that
$$
P_{\omega,-\alpha}(A_1,\dots,A_n)\le G_\omega(A_1,\dots,A_n)
\le P_{\omega,\alpha}(A_1,\dots,A_n),\qquad0<\alpha\le1,
$$
and moreover, as $\alpha\searrow0$,
\begin{align}
P_{\omega,\alpha}(A_1,\dots,A_n)\ &\searrow\ G_\omega(A_1,\dots,A_n), \label{F-2.3}\\
P_{\omega,-\alpha}(A_1,\dots,A_n)\ &\nearrow\ G_\omega(A_1,\dots,A_n). \label{F-2.4}
\end{align}
From these we easily see that $G_\omega$ satisfies (III). The other properties of $G_\omega$
are well-known \cite{LL2,Pa}.
\end{itemize}
\end{example}

\begin{remark}\label{R-2.5}\rm
In \cite{Pa} P\'alfia introduced a generalized notion of operator means of probability
measures on $\bP$ determined by the \emph{generalized Karcher equation}. Restricted to the
$n$-variable situation, the equation is given as
\begin{align}\label{F-2.5}
\sum_{j=1}^nw_jg(X^{-1/2}A_jX^{-1/2})=0,
\end{align}
where $\omega$ is a probability vector and $g$ is an operator monotone function on $(0,\infty)$
with $g(1)=0$ and $g'(1)=1$. For an operator mean $\sigma$ ($\ne\frak{l}$) let
$g_\sigma(x):=(f_\sigma(x)-1)/f_\sigma'(1)$; then $g_\sigma$ is operator monotone on
$(0,\infty)$ satisfying $g_\sigma(1)=0$ and $g_\sigma'(1)=1$, and equation \eqref{F-2.2} for
$M=\cA_\omega$ is equivalent to \eqref{F-2.5} with $g=g_\sigma$. Hence, in the special case
$M=\cA_\omega$, a deformed mean $(\cA_w)_\sigma$ is a generalized operator mean
determined by \eqref{F-2.5}. However, the converse is not true; indeed, the $2$-variable
geometric mean $\#_\alpha$ for $\alpha\in(0,1)$ with $\alpha\ne1/2$ cannot appear
as a deformed mean of the arithmetic mean (see Example \ref{E-2.8}\,(1) below). 
The unique existence of the solution of \eqref{F-2.5} in \cite{Pa} is based on the Banach
contraction principle with respect to $\dT$, while our proof of Theorem \ref{T-2.1} is based
on the monotone continuity of $M$ and the fact given in Lemma \ref{L-2.2}.
\end{remark}

\begin{remark}\label{R-2.6}\rm
A result similar to Theorem \ref{T-2.1}\,(3) for the Karcher mean $G_\omega$ was given in
\cite[Theorem 1]{Ya1}, which says that if $\sum_{j=1}^nw_j\log(X^{-1/2}A_jX^{-1/2})\le0$
then $G_\omega(A_1,\dots,A_n)\le X$ (the same holds with $\ge$ in place of $\le$). This was
useful in the proof of Ando-Hiai's inequality in \cite{Ya1}. Moreover, the same result was
shown in \cite[Theorem 3]{Ya2} for generalized operator means in Remark \ref{R-2.5} to extend
Ando-Hiai's inequality to them. Note also that Theorem \ref{T-2.1}\,(3) for the power mean
(see Example \ref{E-2.4}\,(b)) was shown in \cite[Theorem 4.4]{FS}.
\end{remark}

Let $\tau$ and $\sigma$ be operator means with $\sigma\ne\frak{l}$. Then one can specialize
the above construction of $M_\sigma$ to $M=\tau$ to have the deformed mean
$\tau_\sigma:\bP\times\bP\to\bP$, that is, for every $A,B\in\bP$, $A\tau_\sigma B$ is a unique
solution $X\in\bP$ to the equation
$$
X=(X\sigma A)\tau(X\sigma B).
$$
One can extend this to $\tau_\sigma:B(\cH)^+\times B(\cH)^+\to B(\cH)^+$ in the usual way as
$$
A\tau_\sigma B:=\lim_{\eps\searrow0}(A+\eps I)\tau_\sigma(B+\eps I)\quad
\mbox{in SOT},\qquad A,B\in B(\cH)^+.
$$
Then the next theorem was proved in \cite{Hi}.

\begin{thm}\label{T-2.7}
The deformed mean $\tau_\sigma$ is again an operator mean (in the sense of Kubo-Ando).
Furthermore, the representing function $f_{\tau_\sigma}$ of $\tau_\sigma$ is determined in
such a way that $x=f_{\tau_\sigma}(t)$ for $t>0$ is a unique solution to
$$
(x\sigma1)\tau(x\sigma t)=x,\qquad x>0,
$$
that is,
\begin{align}\label{F-2.6}
f_\sigma(1/x)f_\tau\biggl({f_\sigma(t/x)\over f_\sigma(1/x)}\biggr)=1,\qquad x>0.
\end{align}
\end{thm}

It is obvious that $\tau_{\frak{r}}=\tau$, $\frak{l}_\sigma=\frak{l}$ and
$\frak{r}_\sigma=\frak{r}$ for all $\tau$ and $\sigma\ne\frak{l}$. It is known \cite{Hi} that
$f_{\tau_\sigma}'(1)=f_\tau'(1)$ and for any $\sigma\ne\frak{l}$ the map
$\tau\mapsto\tau_\sigma$ is injective on the operator means. The following examples may be
instructive to understand deformed means for $2$-variable operator means.

\begin{example}\label{E-2.8}\rm
Let $0<w<1$, and consider deformed means $\tau_\sigma$ in the case where $\tau$ is the weighted
arithmetic or the geometric means.
\begin{itemize}
\item[(1)] Let $\tau=\triangledown_w$, the weighted arithmetic mean with the representing
function $1-w+wx$; then equation \eqref{F-2.6} becomes
\begin{align}\label{F-2.7}
(1-w)f_\sigma(1/x)+wf_\sigma(t/x)=1,\qquad x>0.
\end{align}
When $\sigma=\#_\alpha$ with the representing function $x^\alpha$ for $0<\alpha\le1$, the
solution to \eqref{F-2.7} is $x=(1-w+wt^\alpha)^{1/\alpha}$, so we confirm that
$(\triangledown_w)_{\#_\alpha}$ is the $2$-variable case of the power mean $P_{w,\alpha}$ in
Example \ref{E-2.4}\,(b). When $\sigma=\,!_\alpha$, the weighted harmonic mean with the
representing function $(1-\alpha+\alpha x^{-1})^{-1}$ for $0<\alpha\le1$, one can solve
\eqref{F-2.7} to find that the representing function of $(\triangledown_w)_{!_\alpha}$ is
$$
f_{(\triangledown_w)_{!_\alpha}}(t)
={\sqrt{((1-\alpha-w)t+w-\alpha)^2+4\alpha(1-\alpha)t}-((1-\alpha-w)t+w-\alpha)
\over2\alpha}
$$
for $t>0$. In particular, note that $\triangledown_!=\#$, where
$\triangledown=\triangledown_{1/2}$ and $!=\triangledown^*$. Here we show that for any
$w\in(0,1)$, $w\ne1/2$, the deformed means from $\triangledown_w$ do not contain the geometric
mean. Since $f_{(\triangledown_w)_\sigma}'(1)=f_{\triangledown_w}'(1)=w$, only $\#_w$ has a
chance to become a deformed mean from $\triangledown_w$. Assume on the contrary that
$\#_w=(\triangledown_w)_\sigma$ for some $\sigma\ne\frak{l}$. Then, since $x=t^w$ is a solution
to \eqref{F-2.7}, we have
$$
(1-w)f_\sigma(t^{-w})+wf_\sigma(t^{1-w})=1,\qquad t>0.
$$
Letting $t\searrow0$ and $t\nearrow\infty$ gives
$$
(1-w)f_\sigma(\infty)+wf_\sigma(0)=(1-w)f_\sigma(0)+wf_\sigma(\infty)=1.
$$
Thanks to $w\ne1/2$, it must follow that $f_\sigma(0)=f_\sigma(\infty)=1$, so $\sigma=\frak{l}$,
a contradiction.
\item[(2)] Let $\tau=\#_w$; then \eqref{F-2.6} becomes
\begin{align}\label{F-2.8}
f(1/x)\biggl({f(t/x)\over f(1/x)}\biggr)^w=1,\quad\mbox{i.e.},\quad
f(t/x)=f(1/x)^{w-1\over w},
\end{align}
from which we find that the representing function of $(\#_w)_\sigma$ is the inverse function of
$$
g_{w,\sigma}(x):=xf_\sigma^{-1}\bigl(f_\sigma(1/x)^{w-1\over w}\bigr),
\qquad x\in(f_\sigma(0),f_\sigma(\infty)).
$$
For instance, when $\sigma=\triangledown_\alpha$ for $0<\alpha\le1$, one can easily see that
the representing function of $(\#_w)_{\triangledown_\alpha}$ is the inverse function of
$$
g_{w,\alpha}(x)
:={x\over\alpha}\bigl[\bigl(1-\alpha+\alpha x^{-1}\bigr)^{w-1\over w}-1+\alpha\bigr],
\qquad x\in(x_0,\infty),
$$
where $x_0:=\alpha/\bigl[(1-\alpha)^{w\over w-1}-1+\alpha\bigr]$ is determined from
$g_{w,\alpha}(x_0)=0$. When $\sigma$ is the weighted harmonic mean
$!_\alpha=\triangledown_\alpha^*$ for $0<\alpha\le1$, $(\#_w)_{!_\alpha}$ is the adjoint of
$(\#_w)_{\triangledown_\alpha}$, whose representing function is the inverse function of
$g_{w,\alpha}(x^{-1})^{-1}$ on $x\in(0,x_0^{-1})$. Moreover, when $\sigma=\#_\alpha$, that
$(\#_w)_{\#_\alpha}=\#_w$ is immediately seen. In particular, when $\sigma=\#$, we have the
explicit form of the representing function of $\#_{\triangledown_\alpha}$ as
$$
f_{\#_{\triangledown_\alpha}}(t)
={(1-\alpha)(1+t)+\sqrt{(1-\alpha)^2(1+t)^2+4\alpha(2-\alpha)t}\over2(2-\alpha)},\qquad t>0.
$$
\end{itemize}
\end{example}

\section{Ando-Hiai type inequalities}

Assume that $M:\bP^n\to\bP$ is an $n$-variable operator mean satisfying (I)--(IV) introduced
in Section 2. For $A_1,\dots,A_n\in\bP$ we consider inequalities of Ando-Hiai type for $M$
as follows:
\begin{align}
M(A_1^r,\dots,A_n^r)&\ge\lambda_{\min}^{r-1}(M(A_1,\dots,A_n))M(A_1,\dots,A_n)
\quad\ \,\mbox{for $r\ge1$}, \label{F-3.1}\\
M^*(A_1^r,\dots,A_n^r)&\le\|M^*(A_1,\dots,A_n)\|_\infty^{r-1}M^*(A_1,\dots,A_n)
\quad\mbox{for $r\ge1$},\label{F-3.2}
\end{align}
and
\begin{align}
M(A_1^r,\dots,A_n^r)&\le\lambda_{\min}^{r-1}(M(A_1,\dots,A_n))M(A_1,\dots,A_n)
\quad\ \,\mbox{for $0<r\le1$}, \label{F-3.3}\\
M^*(A_1^r,\dots,A_n^r)&\ge\|M^*(A_1,\dots,A_n)\|_\infty^{r-1}M^*(A_1,\dots,A_n)
\quad\mbox{for $0<r\le1$}.\label{F-3.4}
\end{align}
Inequalities \eqref{F-3.1} and \eqref{F-3.2} are the direct generalization of the original
Ando-Hiai inequality in \eqref{F-1.1}. Their weaker and conventional formulations are
\begin{align}
M(A_1,\dots,A_n)&\ge I\ \implies\ M(A_1^r,\dots,A_n^r)\ge I,
\quad\ r\ge1, \label{F-3.5}\\
M^*(A_1,\dots,A_n)&\le I\ \implies\ M^*(A_1^r,\dots,A_n^r)\le I,
\quad r\ge1. \label{F-3.6}
\end{align}
On the other hand, inequalities \eqref{F-3.3} and \eqref{F-3.4} are complementary to
\eqref{F-3.1} and \eqref{F-3.2}, whose original version for the $2$-variable $\#_\alpha$
was given in \cite{Se}. We may write the weaker formulation of \eqref{F-3.3} as
$$
M(A_1,\dots,A_n)\le I\ \implies
\ M(A_1^r,\dots,A_n^r)\le\lambda_{\min}^{r-1}(M(A_1,\dots,A_n))I,
\quad0<r\le1,
$$
and similarly for \eqref{F-3.4}, while they are not so attractive as \eqref{F-3.5} and
\eqref{F-3.6}. It is immediate to see that
$$
\mbox{\eqref{F-3.1}}\,\iff\,\mbox{\eqref{F-3.2}},\quad
\mbox{\eqref{F-3.3}}\,\iff\,\mbox{\eqref{F-3.4}},\quad
\mbox{\eqref{F-3.5}}\,\iff\,\mbox{\eqref{F-3.6}}.
$$

Following \cite{Wa} we say that an operator mean $\sigma$ (in the sense of Kubo-Ando) is
\emph{power monotone increasing} (\emph{p.m.i.}\ for short) if $f_\sigma(x^r)\ge f_\sigma(x)^r$
for all $x>0$ and $r\ge1$. It is clear that the adjoint $\sigma^*$ is p.m.i.\ if and only if
$f_\sigma$ satisfies the reversed inequality. Ando-Hiai's inequality was extended in \cite{Wa}
in such a way that $\sigma$ is p.m.i.\ if and only if
$A\sigma B\ge I$\,$\implies$\,$A^r\sigma B^r\ge I$ holds for all $r\ge1$ and $A,B\in\bP$.

Our main result of this section is the following:

\begin{thm}\label{T-3.1}
Let $M:\bP^n\to\bP$ be an $n$-variable operator mean as above. Let $\sigma$ be a
p.m.i.\ operator mean with $\sigma\ne\frak{l}$. If $M$ satisfies \eqref{F-3.1} (resp.,
\eqref{F-3.3}) for every $A_1,\dots,A_n\in\bP$, then $M_\sigma$ does the same.
\end{thm}

\begin{proof}
Assume that $M$ satisfies \eqref{F-3.1}. First, assume that $1\le r\le2$, and for any
$A_1,\dots,A_n\in\bP$ let $X:=M_\sigma(A_1,\dots,A_n)$ and $\lambda:=\lambda_{\min}(X)$.
By \eqref{F-2.2} we have
\begin{align}\label{F-3.7}
I=M\bigl(f_\sigma(X^{-1/2}A_1X^{-1/2}),\dots,f_\sigma(X^{-1/2}A_nX^{-1/2})\bigr).
\end{align}
To prove \eqref{F-3.1} for $M_\sigma$, by Theorem \ref{T-2.1}\,(3) it suffices to prove that
$$
\lambda^{r-1}X\le M\bigl((\lambda^{r-1}X)\sigma A_1^r,\dots,
(\lambda^{r-1}X)\sigma A_n^r\bigr),
$$
or equivalently,
\begin{align}\label{F-3.8}
I\le M\bigl(f_\sigma(\lambda^{1-r}X^{-1/2}A_1^rX^{-1/2}),\dots,
f_\sigma(\lambda^{1-r}X^{-1/2}A_n^rX^{-1/2})\bigr).
\end{align}
Since $\lambda^{1/2}X^{-1/2}=(\|X^{-1}\|_\infty^{-1}X^{-1})^{1/2}\le I$ and $x^r$ is
operator convex, it follows from Hansen-Pedersen's inequality \cite{HP} that
$$
(\lambda^{1/2}X^{-1/2})A_j^r(\lambda^{1/2}X^{-1/2})
\ge\bigl[(\lambda^{1/2}X^{-1/2})A_j(\lambda^{1/2}X^{-1/2})\bigr]^r
$$
so that
$$
\lambda^{1-r}X^{-1/2}A_j^rX^{-1/2}\ge(X^{-1/2}A_jX^{-1/2})^r.
$$
Therefore,
$$
f_\sigma(\lambda^{1-r}X^{-1/2}A_j^rX^{-1/2})
\ge f_\sigma((X^{-1/2}A_jX^{-1/2})^r)
\ge f_\sigma(X^{-1/2}A_jX^{-1/2})^r,
$$
where the latter inequality follows from p.m.i.\ of $\sigma$. From the monotonicity property
of $M$ and the assumption that $M$ satisfies \eqref{F-3.1} one has
\begin{align*}
&M\bigl(f_\sigma(\lambda^{1-r}X^{-1/2}A_1^rX^{-1/2}),\dots,
f_\sigma(\lambda^{1-r}X^{-1/2}A_n^rX^{-1/2})\bigr) \\
&\quad\ge M\bigl(f_\sigma(X^{-1/2}A_1X^{-1/2})^r,\dots,
f_\sigma(X^{-1/2}A_nX^{-1/2})^r\bigr) \\
&\quad\ge\lambda_{\min}^{r-1}\bigl(M\bigl(f_\sigma(X^{-1/2}A_1X^{-1/2}),\dots,
f_\sigma(X^{-1/2}A_nX^{-1/2})\bigr)\bigr) \\
&\qquad\times M\bigl(f_\sigma(X^{-1/2}A_1X^{-1/2}),\dots,
f_\sigma(X^{-1/2}A_nX^{-1/2})\bigr) \\
&\quad=\lambda_{\min}^{r-1}(I)I=I
\end{align*}
thanks to \eqref{F-3.7}. Hence \eqref{F-3.8} follows.

To prove \eqref{F-3.1} for general $r\ge1$, we use a simple induction argument. Assume that
\eqref{F-3.1} is true when $1\le r\le 2^k$, and extend it to $1\le r\le2^{k+1}$. When
$2^k\le r\le2^{k+1}$, letting $r=2r'$ with $1\le r'\le2^k$ one has
\begin{align*}
M_\sigma(A_1^r,\dots,A_n^r)
&\ge\lambda_{\min}(M_\sigma(A_1^{r'},\dots,A_n^{r'}))M(A_1^{r'},\dots,A_n^{r'}) \\
&\ge\lambda_{\min}\bigl(\lambda_{\min}^{r'-1}(M(A_1,\dots,A_n))M(A_1,\dots,A_n)\bigr) \\
&\quad\times\lambda_{\min}^{r'-1}(M(A_1,\dots,A_n))M(A_1,\dots,A_n) \\
&=\lambda_{\min}^{r-1}(M(A_1,\dots,A_n))M(A_1,\dots,A_n).
\end{align*}

The other assertion for \eqref{F-3.3} can be proved in a way similar to the above proof for
\eqref{F-3.1}, by reversing all the inequality signs and using Hansen's inequality \cite{Ha}
for the operator monotone function $x^r$ when $0<r\le1$. An additional argument in the last
of the above proof is unnecessary. The details may be omitted.
\end{proof}

Recently in \cite{Ya2}, Yamazaki proved two Ando-Hiai type inequalities for $n$-variable
generalized operator means in the sense of \cite{Pa}. For a probability vector $\omega$ and an
operator monotone function $g$ on $(0,\infty)$ with $g(1)=0$ and $g'(1)=1$, let
$\Lambda_{\omega,g}$ be the $n$-variable operator mean satisfying the equation given in
\eqref{F-2.5}. It was proved in \cite[Theorem 7]{Ya2} that the following conditions are
equivalent:
\begin{itemize}
\item[(a)] $\Lambda_{\omega,g}(A_1,\dots,A_n)\ge I$ $\implies$
$\Lambda_{\omega,g}(A_1^r,\dots,A_n^r)\ge I$, $r\ge1$, for all probability vectors
$\omega=(w_1,\dots,w_n)$ and $A_1,\dots,A_n\in\bP$;
\item[(b)] $g(x^r)\ge rg(x)$ for all $r\ge1$ and $x>0$.
\end{itemize}

Although \cite[Theorem 7]{Ya2} gives Ando-Hiai's inequalities, in particular, for the power
means $P_{\omega,\alpha}$ as well as the Karcher mean $G_\omega$, we independently state
corollaries for those means in the following, because our inequalities are in the stronger
form of \eqref{F-3.1} and \eqref{F-3.2} and include the complementary version of
\eqref{F-3.3} and \eqref{F-3.4} as well. We first give a simple lemma, whose proof is given
for completeness.

\begin{lemma}\label{L-3.2}
Inequalities \eqref{F-3.1} and \eqref{F-3.3} hold for $M=\cA_\omega$, and \eqref{F-3.2} and
\eqref{F-3.4} hold for $M^*=\cH_\omega$.
\end{lemma}

\begin{proof}
Let $A_1,\dots,A_n\in\bP$. When $1\le r\le2$, operator convexity of $x^r$ gives
$$
\sum_{j=1}^nw_jA_j^r\ge\Biggl(\sum_{j=1}^nw_jA_j\Biggr)^r
\ge\lambda_{\min}^{r-1}\Biggl(\sum_{j=1}^nw_jA_j\Biggr)\sum_{j=1}^nw_jA_j.
$$
By an induction argument as in the proof of Theorem \ref{T-3.1} one can extend the above
inequality to general $r\ge1$. Moreover, when $0<r\le1$, since $x^r$ is operator concave and
$X\ge\lambda_{\min}^{1-r}(X)X^r$ for $X\in\bP$, one has
$$
\sum_{j=1}^nw_jA_j^r\le\Biggl(\sum_{j=1}^nw_jA_j\Biggr)^r
\le\lambda_{\min}^{r-1}\Biggl(\sum_{j=1}^nw_jA_j\Biggr)\sum_{j=1}^nw_jA_j.
$$
Hence \eqref{F-3.1} and \eqref{F-3.3} hold for $M=\cA_\omega$. The latter assertion follows
since $\cH_\omega=\cA_\omega^*$.
\end{proof}

In the following corollaries let $\omega=(w_1,\dots,w_n)$ be any probability vector.

\begin{cor}\label{C-3.3}
Let $0<\alpha\le1$ and $A_1,\dots,A_n\in\bP$. For any $r\ge1$,
\begin{align}
P_{\omega,\alpha}(A_1^r,\dots,A_n^r)
&\ge\lambda_{\min}^{r-1}(P_{\omega,\alpha}(A_1,\dots,A_n))P_{w,\alpha}(A_1,\dots,A_n),
\label{F-3.9}\\
P_{\omega,-\alpha}(A_1^r,\dots,A_n^r)
&\le\|P_{\omega,-\alpha}(A_1,\dots,A_n)\|_\infty^{r-1}P_{w,-\alpha}(A_1,\dots,A_n),
\label{F-3.10}
\end{align}
and for any $r\in(0,1]$,
\begin{align}
P_{\omega,\alpha}(A_1^r,\dots,A_n^r)
&\le\lambda_{\min}^{r-1}(P_{\omega,\alpha}(A_1,\dots,A_n))P_{w,\alpha}(A_1,\dots,A_n),
\label{F-3.11}\\
P_{\omega,-\alpha}(A_1^r,\dots,A_n^r)
&\ge\|P_{\omega,-\alpha}(A_1,\dots,A_n)\|_\infty^{r-1}P_{w,-\alpha}(A_1,\dots,A_n).
\label{F-3.12}
\end{align}
\end{cor}

\begin{proof}
Note that $\#_\alpha$ is p.m.i., $P_{\omega,\alpha}=(\cA_\omega)_{\#_\alpha}$ and
$P_{\omega,-\alpha}=P_{\omega,\alpha}^*$ for $0<\alpha\le1$. Hence the corollary follows from
Theorem \ref{T-3.1} and Lemma \ref{L-3.2}.
\end{proof}

\begin{cor}\label{C-3.4}
Let $A_1,\dots,A_n\in\bP$. For any $r\ge1$,
\begin{align}
&\lambda_{\min}^{r-1}(G_\omega(A_1,\dots,A_n))G_\omega(A_1,\dots,A_n) \nonumber\\
&\qquad\le G_\omega(A_1^r,\dots,A_n^r)
\le\|G_\omega(A_1,\dots,A_n)\|_\infty^{r-1}G_\omega(A_1,\dots,A_n), \label{F-3.13}
\end{align}
and for any $r\in(0,1]$,
\begin{align}
&\|G_\omega(A_1,\dots,A_n)\|_\infty^{r-1}G_\omega(A_1,\dots,A_n) \nonumber\\
&\qquad\le G_\omega(A_1^r,\dots,A_n^r)
\le\lambda_{\min}^{r-1}(G_\omega(A_1,\dots,A_n))G_\omega(A_1,\dots,A_n). \label{F-3.14}
\end{align}
\end{cor}

\begin{proof}
For every $A_1,\dots,A_n\in\bP$, from \eqref{F-2.3} and \eqref{F-2.4} one can easily verify
that as $\alpha\searrow0$,
\begin{align*}
\lambda_{\min}(P_{\omega,\alpha}(A_1,\dots,A_n))
\ &\searrow\ \lambda_{\min}(G_\omega(A_1,\dots,A_n)), \\
\|P_{\omega,-\alpha}(A_1,\dots,A_n)\|_\infty
\ &\nearrow\ \|G_\omega(A_1,\dots,A_n)\|_\infty.
\end{align*}
Hence the corollary follows by taking the limits of the inequalities in
\eqref{F-3.9}--\eqref{F-3.12}.
\end{proof}

Ando-Hiai's inequality for the Karcher mean in \eqref{F-3.13} was proved in \cite{Ya1} for
positive definite matrices, which was further extended in \cite{KLL,HL} to the Cartan
barycenter of probability measures.

\begin{remark}\label{R-3.5}\em
One might think that the complementary inequality in \eqref{F-3.14} can follow from
\eqref{F-3.13} by replacing $r$ with $1/r$ and $A_j$ with $A_j^r$ in \eqref{F-3.13}. But this
is not the case, as seen below. Let $0<r\le1$ and replace $r$ with $1/r$ and $A_j$ with
$A_j^r$ in the second inequality of \eqref{F-3.13}. Then
$$
\|G_\omega(A_1^r,\dots,A_n^r)\|_\infty^{1-{1\over r}}G_\omega(A_1,\dots,A_n)
\le G_\omega(A_1^r,\dots,A_n^r).
$$
Although we have $\|G_\omega(A_1,\dots,A_n)\|_\infty^r\le\|G_\omega(A_1^r,\dots,A_m^r)\|_\infty$
from the weaker formulation of \eqref{F-3.13}, this argument cannot, due to $1-{1\over r}\le0$,
give the first inequality of \eqref{F-3.14}. However, note that the opposite direction works
well, that is, \eqref{F-3.13} follows from \eqref{F-3.14} by replacing $r$ with $1/r$ and $A_j$
with $A_j^r$. So \eqref{F-3.14} is new and considered as stronger than \eqref{F-3.13}.
\end{remark}

The next theorem says that an operator mean $\sigma$ is p.m.i.\ if $(G_\omega)_\sigma$
satisfies Ando-Hiai's inequality for any probability vector $\omega$, thus showing that the
p.m.i.\ assumption on $\sigma$ in Theorem \ref{T-3.1} is essential. Conditions (a) and (b)
above are incorporated as (v) and (vi) in the theorem.

\begin{thm}\label{T-3.6}
Let $\sigma$ be an operator mean with $\sigma\ne\frak{l}$, and $f:=f_\sigma$. Consider the
following conditions, where {\rm(ii)}--{\rm(v)} mean that the condition holds for all
probability vectors $\omega$ and all $A_1,\dots,A_n\in\bP$:
\begin{itemize}
\item[\rm(i)] $\sigma$ is p.m.i., i.e., $f(x^r)\ge f(x)^r$ for all $r\ge1$ and $x>0$;
\item[\rm(ii)] $(G_\omega)_\sigma(A_1^r,\dots,A_n^r)\ge
\lambda_{\min}^{r-1}((G_\omega)_\sigma(A_1,\dots,A_n))(G_\omega)_\sigma(A_1,\dots,A_n)$ for
all $r\ge1$;
\item[\rm(iii)] $(G_\omega)_\sigma(A_1,\dots,A_n)\ge I$ $\implies$
$(G_\omega)_\sigma(A_1^r,\dots,A_n^r)\ge I$, $r\ge1$;
\item[\rm(iv)] $(\cA_\omega)_{\sigma_{f^p}}(A_1,\dots,A_n)\ge I$ $\implies$
$(\cA_\omega)_{\sigma_{f^p}}(A_1^r,\dots,A_n^r)\ge I$, $r\ge1$, for every $p\in(0,1]$,
where $\sigma_{f^p}$ is the operator mean defined by the operator
monotone function $f(x)^p$;
\item[\rm(v)] $(\cA_\omega)_\sigma(A_1,\dots,A_n)\ge I$ $\implies$
$(\cA_\omega)_\sigma(A_1^r,\dots,A_n^r)\ge I$, $r\ge1$;
\item[$(${\rm vi}$)$] $f(x^r)\ge rf(x)-r+1$ for all $r\ge1$ and $x>0$.
\end{itemize}
Then the following relations hold:
$$
({\rm i}) \iff ({\rm ii}) \iff ({\rm iii}) \iff ({\rm iv}) \implies ({\rm v}) \iff ({\rm vi}).
$$
\end{thm}

\begin{proof}
(i)\,$\implies$\,(ii) is seen from Theorem \ref{T-3.1} and Corollary \ref{C-3.4}.

(ii)\,$\implies$\,(iii) is obvious.

(iii)\,$\implies$\,(i).\enspace
For every $w\in(0,1]$, apply (iii) to $A\#_wB=G_\omega(A,\dots,A,B)$ where
$\omega=((1-w)/(n-1),\dots,(1-w)/(n-1),w)$; then we see that for any $r\ge1$,
$$
A,B\in\bP,\ \ A(\#_w)_\sigma B\ge I \implies A^r(\#_w)_\sigma B^r\ge I.
$$
Hence \cite[Lemma 2.1]{Wa} implies that the representing function $f_w$ of $(\#_w)_\sigma$ is
p.m.i. As mentioned in Example \ref{E-2.8}\,(2), note that for any $t>0$, $x=f_w(t)$ is the
solution to \eqref{F-2.8}. Since $f'(1)>0$ due to $\sigma\ne\frak{l}$, one can write
$$
(1-w){1\over f'(1)}\log f\biggl({1\over f_w(t)}\biggr)
+w\,{1\over f'(1)}\log f\biggl({t\over f_w(t)}\biggr)=0,\qquad t>0.
$$
This is the equation of \cite[(3.1)]{Ya2} for $g(x):={1\over f'(1)}\log f(x)$, an operator
monotone function on $(0,\infty)$ with $g(1)=0$ and $g'(1)=1$. Hence \cite[Theorem 7]{Ya2}
implies that $g(x^r)\ge rg(x)$, i.e., $f(x^r)\ge f(x)^r$ for all $r\ge1$ and $x>0$.

(i)\,$\implies$\,(iv).\enspace
If $\sigma$ is p.m.i., then so is $\sigma_{f^p}$ for any $p\in(0,1]$. Hence (iv) follows from
Theorem \ref{T-3.1} and Lemma \ref{L-3.2} (or from \cite[Theorem 7]{Ya2}).

(iv)\,$\implies$\,(v) is obvious.

(v)\,$\iff$\,(vi).\enspace
Letting $g(x):=(f(x)-1)/f'(1)$ for $x>0$, note that $(\cA_\omega)_\sigma=\Lambda_{\omega,g}$,
where $\Lambda_{\omega,g}$ is as in condition (a) above (see also Remark \ref{R-2.5}).
Hence the equivalence of (v) and (vi) reduces to that of (a) and (b) in \cite[Theorem 7]{Ya2}.

(iv)\,$\implies$\,(i).\enspace
By (v)\,$\iff$\,(vi), condition (iv) implies that $f(x^r)^p\ge rf(x)^p-r+1$ for all $r\ge1$,
$x>0$ and $p\in(0,1]$. Hence one has
$$
{f(x^r)^p-1\over p}\ge r\,{f(x)^p-1\over p}.
$$
Letting $p\searrow0$ gives $\log f(x^r)\ge r\log f(x)$, implying (i).
\end{proof}

\begin{example}\label{E-3.7}\rm
We here supply an example of operator monotone functions $f$ on $(0,\infty)$ with $f(1)=1$
showing that (i) is strictly stronger than (vi) in Theorem \ref{T-3.6}. Let
$$
f(x):={1\over4}\biggl({x+1\over2}\biggr)+{3\over4}\biggl({2x\over x+1}\biggr),
$$
a convex combination of the representing functions of the arithmetic and the harmonic means. Since
$$
f(2^3)={59\over24}<{1331\over512}=f(2)^3,
$$
the function $f$ does not satisfy (i). On the other hand, compute
$$
xf''(x)+f'(x)={x^3+3x^2-9x+13\over8(x+1)^3}>0,\qquad x>0.
$$
For any fixed $x>0$ we have
$$
{d^2\over dr^2}\,f(x^r)={d\over dr}(f'(x^r)x^r\log x)
=(x^rf''(x^r)+f'(x^r))x^r(\log x)^2\ge0,\qquad r>0,
$$
which means that $r\in(0,\infty)\mapsto f(x^r)$ is a convex function. Therefore, for every
$r>1$,
$$
{f(x^r)-f(x)\over r-1}\ge f(x)-1,
$$
showing that $f$ satisfies (vi).
\end{example}

\begin{remark}\label{R-3.8}\rm
Assume that $M:\bP^n\to\bP$ is an $n$-variable operator mean and satisfies
\begin{align}\label{F-3.15}
\cH_\omega\le M\le\cA_\omega
\end{align}
for some probability vector $\omega=(w_1,\dots,w_n)$. Then for any operator mean $\sigma$ with
$\alpha:=f_\sigma'(1)\in(0,1]$, since $\cH_\omega=(\cH_\omega)_{!_\alpha}\le M_\sigma\le
(\cA_\omega)_{\triangledown_\alpha}=\cA_\omega$ as easily shown by use of Theorem
\ref{T-2.1}\,(3), it follows that $M_\sigma$ and hence $M_\sigma^*$ satisfy \eqref{F-3.15}
again, which gives the Lie-Trotter formula
\begin{align}\label{F-3.16}
\lim_{p\to0}M_\sigma(A_1^p,\dots,A_n^p)^{1/p}=\exp\Biggl(\sum_{j=1}^nw_j\log A_j\Biggr)
\end{align}
for every $A_1,\dots,A_n\in\bP$ (as in, e.g., \cite{FFS,HL}). Now, assume further that $M$
satisfies \eqref{F-3.1} and $\sigma$ is a p.m.i.\ operator mean with $\sigma\ne\frak{l}$. Since
\eqref{F-3.2} holds for $M_\sigma^*$ by Theorem \ref{T-3.1}, we see by \eqref{F-3.16}
for $M_\sigma^*$ that for every $A_1,\dots,A_n\in\bP$,
$\big\|M_\sigma^*(A_1^p,\dots,A_n^p)\big\|_\infty^{1/p}$ increases to
$\big\|\exp\bigl(\sum_{j=1}^nw_j\log A_j\bigr)\big\|_\infty$ as $p\searrow0$. In particular,
this holds for $M_\sigma^*=P_{\omega,\alpha}$ for $0<\alpha\le1$ when applied to $M=\cA_\omega$
and $\sigma=\#_\alpha$. When $A_1,\dots,A_n$ are positive definite matrices, it also follows
from $P_{\omega,-\alpha}\le G_\omega$ and the log-majorization \cite[(3.20)]{HiPe} that
$$
\UIN{P_{\omega,-\alpha}(A_1^p,\dots,A_n^p)^{1/p}}
\le\UIN{G_\omega(A_1^p,\dots,A_n^p)^{1/p}}
\le\UIN{\exp\Biggl(\sum_{j=1}^nw_j\log A_j\Biggr)}
$$
for any unitarily invariant norm $\UIN{\,\cdot\,}$.
\end{remark}

\section{Modified inequalities}

In this section assume that $M$ is a general $n$-variable operator mean and $\sigma$ is any
operator mean with $\sigma\ne\frak{l}$. We present two more inequalities of Ando-Hiai type
for the deformed mean $M_\sigma$ without the assumption of p.m.i.\ for $\sigma$.

\begin{thm}\label{T-4.1}
For every $A_1,\dots,A_n\in\bP$ and any $r\ge1$,
\begin{align}
&\lambda_{\min}^{r-1}(M_\sigma(A_1,\dots,A_n))M_\sigma(A_1,\dots,A_n) \nonumber\\
&\qquad\le M_{\sigma_{1/r}}(A_1^r,\dots,A_n^r)
\le\|M_\sigma(A_1,\dots,A_n)\|_\infty^{r-1}M_\sigma(A_1,\dots,A_n), \label{F-4.1}
\end{align}
where $\sigma_{1/r}$ is the operator mean with the representing function $f_\sigma(x^{1/r})$. 
\end{thm}

\begin{proof}
Let $A_1,\dots,A_n\in\bP$ and $r\ge1$. The first inequality in \eqref{F-4.1} follows from the
second inequality. Indeed, replace $M$, $\sigma$ and $A_j$ in the second with $M^*$,
$\sigma^*$ and $A_j^{-1}$; then since $(\sigma^*)_{1/r}=(\sigma_{1/r})^*$ so that
$(M^*)_{(\sigma^*)_{1/r}}=(M_{\sigma_{1/r}})^*$ as well as $(M^*)_{\sigma^*}=(M_\sigma)^*$,
we have
$$
M_{\sigma_{1/r}}(A_1^r,\dots,A_n^r)^{-1}
\le\|M_\sigma(A_1,\dots,A_n)^{-1}\|_\infty^{r-1}M_\sigma(A_1,\dots,A_n)^{-1},
$$
which gives
$$
M_{\sigma_{1/r}}(A_1^r,\dots,A_n^r)
\ge\lambda_{\min}^{r-1}(M_\sigma(A_1,\dots,A_n))M_\sigma(A_1,\dots,A_n).
$$
So we may prove the second inequality only. Let $X:=M_\sigma(A_1,\dots,A_n)$; then we have
\eqref{F-3.7}. By Theorem \ref{T-2.1}\,(3) it suffices to prove that
$$
\|X\|_\infty^{r-1}X
\ge M\bigl(\bigl(\|X\|_\infty^{r-1}X\bigr)\sigma_{1/r}A_1^r,\dots,
\bigl(\|X\|_\infty^{r-1}X\bigr)\sigma_{1/r}A_n^r\bigr),
$$
or equivalently,
$$
I\ge M\bigl(
f_\sigma\bigl(\|X\|_\infty^{{1\over r}-1}(X^{-1/2}A_1^rX^{-1/2})^{1/r}\bigr),
\dots,f_\sigma\bigl(\|X\|_\infty^{{1\over r}-1}(X^{-1/2}A_n^rX^{-1/2})^{1/r}
\bigr)\bigr).
$$
Since $C:=\|X\|_\infty^{-1/2}X^{1/2}\le I$, Hansen's inequality \cite{Ha} gives
$C(C^{-1}A_j^rC^{-1})^{1/r}C\le A_j$ so that $(C^{-1}A_j^rC^{-1})^{1/r}\le C^{-1}A_jC^{-1}$,
which implies that
$$
\|X\|_\infty^{{1\over r}-1}(X^{-1/2}A_j^rX^{-1/2})^{1/r}\le X^{-1/2}A_jX^{-1/2}.
$$
Therefore, it follows from the monotonicity of $f_\sigma$ and $M$ that
\begin{align*}
&M\bigl(f_\sigma\bigl(\|X\|_\infty^{{1\over r}-1}(X^{-1/2}A_1^rX^{-1/2})^{1/r}\bigr),
\dots,f_\sigma\bigl(\|X\|_\infty^{{1\over r}-1}(X^{-1/2}A_n^rX^{-1/2})^{1/r}
\bigr)\bigr) \\
&\qquad\le M\bigl(f_\sigma(X^{-1/2}A_1X^{-1/2}),\dots,
f_\sigma(X^{-1/2}A_nX^{-1/2})\bigr)=I
\end{align*}
thanks to \eqref{F-3.7}. Hence the second inequality of \eqref{F-4.1} follows.
\end{proof}

Remark that Theorem \ref{T-4.1} is similar to \cite[Theorem 6]{Ya2} while $M$ in Theorem
\ref{T-4.1} is a general $n$-variable operator mean and our inequality is in the stronger form
of \eqref{F-3.1} and \eqref{F-3.2}.

\begin{thm}\label{T-4.2}
For every $A_1,\dots,A_n\in\bP$ and any $r\in(0,1]$,
\begin{align}
&\|M_{\sigma_r}(A_1,\dots,A_n)\|_\infty^{r-1}M_{\sigma_r}(A_1,\dots,A_n) \nonumber\\
&\qquad\le M_\sigma(A_1^r,\dots,A_n^r)
\le\lambda_{\min}^{r-1}(M_{\sigma_r}(A_1,\dots,A_n))M_{\sigma_r}(A_1,\dots,A_n), \label{F-4.2}
\end{align}
where $\sigma_r$ is the operator mean with the representing function $f_\sigma(x^r)$.
\end{thm}

\begin{proof}
Similarly to \eqref{F-4.1} the first inequality in \eqref{F-4.2} follows from the second, so
we may prove the latter only. Assume that $0<r\le1$. Let $X:=M_{\sigma_r}(A_1,\dots,A_n)$
and $\lambda:=\lambda_{\min}(X)$. We have
\begin{align}\label{F-4.3}
I=M\bigl(f_{\sigma_r}(X^{-1/2}A_1X^{-1/2}),\dots,f_{\sigma_r}(X^{-1/2}A_nX^{-1/2})\bigr).
\end{align}
By Theorem \ref{T-2.1}\,(3) it suffices to prove that
$$
\lambda^{r-1}X\ge M\bigl((\lambda^{r-1}X)\sigma A_1^r,\dots,
(\lambda^{r-1}X)\sigma A_n^r\bigr),
$$
or equivalently,
$$
I\ge M\bigl(f_\sigma(\lambda^{1-r}X^{-1/2}A_1^rX^{-1/2}),\dots,
f_\sigma(\lambda^{1-r}X^{-1/2}A_n^rX^{-1/2})\bigr).
$$
Since $\lambda^{1/2}X^{-1/2}\le I$, it follows from Hansen's inequality \cite{Ha} that
$$
\lambda X^{-1/2}A_j^rX^{-1/2}\le(\lambda X^{-1/2}A_jX^{-1/2})^r,
$$
and hence $\lambda^{1-r}X^{-1/2}A_j^rX^{-1/2}\le(X^{-1/2}A_jX^{-1/2})^r$. Therefore,
\begin{align*}
&M\bigl(f_\sigma(\lambda^{1-r}X^{-1/2}A_1^rX^{-1/2}),\dots,
f_\sigma(\lambda^{1-r}X^{-1/2}A_n^rX^{-1/2})\bigr) \\
&\quad\le M\bigl(f_\sigma((X^{-1/2}A_1X^{-1/2})^r),\dots,
f_\sigma((X^{-1/2}A_nX^{-1/2})^r)\bigr) \\
&\quad=M\bigl(f_{\sigma_r}(X^{-1/2}A_1X^{-1/2}),\dots,
f_{\sigma_r}(X^{-1/2}A_nX^{-1/2})\bigr)=I
\end{align*}
thanks to \eqref{F-4.3}.
\end{proof}

Inequalities \eqref{F-4.1} and \eqref{F-4.2} are modifications of \eqref{F-3.1}--\eqref{F-3.4}
in the previous section, where $M_\sigma$ in the either left or right side is replaced with
$M_{\sigma_{1/r}}$ in \eqref{F-4.1} or $M_{\sigma_r}$ in \eqref{F-4.2}. On the other hand,
there are no restrictions on $M$ and $\sigma$ in \eqref{F-4.1} and \eqref{F-4.2}, while we set
additional assumptions on $M$ and $\sigma$ in Theorem \ref{T-3.1}.

\begin{remark}\label{R-4.3}\rm
Let $r\ge1$ and $A_1,\dots,A_n\in\bP$. Theorems \ref{T-4.1} and \ref{T-4.2} in particular
contain the following implications:
\begin{align*}
M_\sigma(A_1,\dots,A_n)\ge I\ &\implies\ M_{\sigma_{1/r}}(A_1^r,\dots,A_n^r)\ge I \\
&\implies\ M_{\sigma_{1/r}}(A_1^r,\dots,A_n^r)\ge M_\sigma(A_1,\dots,A_n),
\end{align*}
and the same implications with $\le$ in place of $\ge$. In fact, the first implication is
obvious from Theorem \ref{T-4.1}, and the second is seen by replacing $r$ with $1/r$ and
$A_j$ with $A_j^r$ in Theorem \ref{T-4.2}.
\end{remark}

Let $\omega$ be any probability vector. When $M=G_\omega$ and $\sigma=\#_\alpha$,
since $(G_\omega)_{\#_\alpha}=G_\omega$ for any $\alpha\in(0,1]$, Theorems \ref{T-4.1} and
\ref{T-4.2} reduce to Corollary \ref{C-3.4}. When $M=\cA_\omega$ or $\cH_\omega$ and
$\sigma=\#_\alpha$, Theorems \ref{T-4.1} and \ref{T-4.2} show the following  inequalities in
\eqref{F-4.4} and \eqref{F-4.5}, respectively. The second inequality in \eqref{F-4.4} when
$0<\alpha\le1$ and the first inequality in \eqref{F-4.4} when $-1\le\alpha<0$ were first
shown in \cite{LY}.

\begin{cor}\label{C-4.3}
Let $\alpha\in[-1,1]\setminus\{0\}$. For any $r\ge1$,
\begin{align}
&\lambda_{\min}^{r-1}(P_{w,\alpha}(A_1,\dots,A_n))P_{w,\alpha}(A_1,\dots,A_n) \nonumber\\
&\qquad\le P_{w,\alpha/r}(A_1^r,\dots,A_n^r)
\le\|P_{w,\alpha}(A_1,\dots,A_n)\|_\infty^{r-1}P_{w,\alpha}(A_1,\dots,A_n), \label{F-4.4}
\end{align}
and for any $r\in(0,1]$,
\begin{align}
&\|P_{w,\alpha r}(A_1,\dots,A_n)\|_\infty^{r-1}P_{w,\alpha r}(A_1,\dots,A_n) \nonumber\\
&\qquad\le P_{w,\alpha}(A_1^r,\dots,A_n^r)
\le\lambda_{\min}^{r-1}(P_{w,\alpha r}(A_1,\dots,A_n))P_{w,\alpha r}(A_1,\dots,A_n).
\label{F-4.5}
\end{align}
\end{cor}

In the rest of the section we apply Theorems \ref{T-4.1} and \ref{T-4.2} to $2$-variable
operator means (in the sense of Kubo-Ando). Even in this specialized situation new
inequalities of Ando-Hiai type show up. Let $\tau$ and $\sigma$ be operator means with
$\sigma\ne\frak{l}$. We have the deformed operator mean $\tau_\sigma$ (see Theorem
\ref{T-2.7} and Example \ref{E-2.8}). Specializing Theorems \ref{T-4.1} and \ref{T-4.2} to
$\tau$ and $\sigma$ we have the following:

\begin{cor}\label{C-4.4}
$(${\rm 1}$)$\enspace
For every $A,B\in\bP$,
\begin{align}
\lambda_{\min}^{r-1}(A\tau_\sigma B)(A\tau_\sigma B)
&\le A^r\tau_{\sigma_{1/r}}B^r\le\|A\tau_\sigma B\|_\infty^{r-1}(A\tau_\sigma B),
\qquad r\ge1, \label{F-4.6}\\
\|A\tau_{\sigma_r}B\|_\infty^{r-1}(A\tau_{\sigma_r}B)
&\le A^r\tau_\sigma B^r\le\lambda_{\min}^{r-1}(A\tau_{\sigma_r}B)(A\tau_{\sigma_r}B),
\qquad0<r\le1, \label{F-4.7}
\end{align}
where $\sigma_r$ for $0<r\le1$ is the operator mean whose representing function is
$f_\sigma(x^r)$.

$(${\rm 2}$)$\enspace
The first inequality of \eqref{F-4.6} and the second one of \eqref{F-4.7} hold for every
$A,B\in B(\cH)^+$. When $\cH$ is finite-dimensional, all the inequalities above extend to
$A,B\in B(\cH)^+$.
\end{cor}

\begin{proof}
(1) is just the special case of Theorems \ref{T-4.1} and \ref{T-4.2}. The first assertion
of (2) follows by taking the limit of the inequalities in question for $A+\eps I$ and
$B+\eps I$ as $\eps\searrow0$, since
$\lambda_{\min}((A+\eps I)\tau_\sigma(B+\eps I))\searrow\lambda_{\min}(A\tau_\sigma B)$.
When $\cH$ is finite-dimensional, we have
$\|(A+\eps I)\tau_\sigma(B+\eps I)\|_\infty\searrow\|A\tau_\sigma B\|_\infty$ as well, so
the latter assertion follows.
\end{proof}

In connection with the above proof we note that
$$
\lim_{\eps\searrow0}\|(A+\eps I)\#(B+\eps I)\|_\infty=\|A\# B\|_\infty
$$
fails to hold for $A,B\in B(\cH)^+$ in the infinite-dimensional case. So it does not seem easy to
extend the second inequality of \eqref{F-4.6} and the first one of \eqref{F-4.7} to
$A,B\in B(\cH)^+$ in the infinite-dimensional case.

For the right trivial mean $\frak{r}$ we have $\tau_{\frak{r}}=\tau$ and $\frak{r}_r=\#_r$ for
every $r\in(0,1]$. Moreover, when $\sigma=\#_r$, equation \eqref{F-2.6} is
$$
(1/x)^rf_\tau\biggl({(t/x)^r\over(1/x)^r}\biggr)=1,\quad\mbox{i.e.,}\quad
f_\tau(t^r)=x^r,
$$
whose solution is $x=f_\tau(t^r)^{1/r}$. We write $\tau_{[r]}$ for the operator mean whose
representing function is $f_\tau(t^r)^{1/r}$, where $0<r\le1$. Then Corollary \ref{C-4.4}
specialized to the case $\sigma=\frak{r}$ is the following:

\begin{cor}\label{C-4.5}
$(${\rm 1}$)$\enspace
For every operator mean $\tau$ and every $A,B\in\bP$,
\begin{align}
\lambda_{\min}^{r-1}(A\tau B)(A\tau B)&\le A^r\tau_{[1/r]}B^r
\le\|A\tau B\|_\infty^{r-1}(A\tau B),\qquad r\ge1, \label{F-4.8}\\
\|A\tau_{[r]}B\|_\infty^{r-1}(A\tau_{[r]}B)&\le A^r\tau B^r
\le\lambda_{\min}^{r-1}(A\tau_{[r]}B)(A\tau_{[r]}B), \qquad 0<r\le1. \label{F-4.9}
\end{align}

$(${\rm 2}$)$\enspace
The first inequality in \eqref{F-4.8} and the second one in \eqref{F-4.9} hold for every
$A,B\in B(\cH)^+$. When $\cH$ is finite-dimensional, all the inequalities above extend to
$A,B\in B(\cH)^+$.
\end{cor}

For example, when $\tau=\#_\alpha$ and hence $\tau_{[r]}=\#_\alpha$ for any $r\in(0,1]$,
\eqref{F-4.8} reduces to the original Ando-Hiai inequality in \eqref{F-1.1} and \eqref{F-4.9}
is its complementary version in \cite{Se}. When $\tau$ is p.m.i.\ and hence $\tau_{[r]}\le\tau$
for any $r\in(0,1]$, \eqref{F-4.8} gives a generalized Ando-Hiai's inequality in \cite{Wa},
and \eqref{F-4.9} gives its complementary version generalizing that in \cite{Se}.

The following is an even more generalized version of the above mentioned inequality in
\cite{Wa}.

\begin{cor}\label{C-4.6}
For each $r\ge1$ and operator means $\sigma,\tau$ the following conditions are equivalent:
\begin{itemize}
\item[\rm(i)] $A,B\in\bP$, $A\tau B\ge I$ $\implies$ $A^r\sigma B^r\ge I$;
\item[\rm(ii)] $f_\sigma(t^r)\ge f_\tau(t)^r$ for all $t>0$.
\end{itemize}
\end{cor}

\begin{proof}
(i)\,$\implies$\,(ii).\enspace
For every $t>0$ let $x:=f_\tau(t)$. Then, since $(1/x)\tau(t/x)=1$, (i) implies that
$(1/x^r)\sigma(t^r/x^r)\ge1$ and hence $1\sigma t^r\ge x^r$, i.e.,
$f_\sigma(t^r)\ge f_\tau(t)^r$.

(ii)\,$\implies$\,(i).\enspace
Assume (ii), which means that $f_\tau(t^{1/r})^r\le f_\sigma(t)$ for all $t>0$, i.e.,
$f_{\tau_{[1/r]}}\le f_\sigma$. Hence, if $A\tau B\ge I$, then by Corollary \ref{C-4.5}\,(1)
we have $I\le A^r\tau_{[1/r]}B^r\le A^r\sigma B^r$
\end{proof}

\section{Reverse inequalities}

In this section we show some reverse inequalities of Ando-Hiai type involving the generalized
Kantorovich constant. For each $h>1$, the \emph{generalized Kantorovich constant} $K(h,p)$ is
defined by
\begin{equation}\label{F-5.1}
K(h,p):={h^p-h\over(p-1)(h-1)}\biggl({p-1\over p}\cdot{h^p-1\over h^p-h}\biggr)^p,
\qquad p\in\bR,
\end{equation}
where $K(h,1)=\lim_{p\to1}K(h,p)=1$. It is known \cite[Lemma 2.4]{FSY} and \cite[Theorem 2.54]{FMPS} that if $p>1$, then $K(h^t,p)^{1/t}$
is increasing in $t>0$ and $1\le K(h^t,p)^{1/t}\le h^{p-1}$ for all $t>0$. Moreover, the
following can easily be verified by applying L'Hospital's rule:
\begin{align}\label{F-5.2}
\lim_{t\to0}K(h^t,p)^{1/t}=1.
\end{align}

\begin{lemma}\label{L-5.1}
Let $A,C\in\bP$ and assume that $mI\le A\le MI$ and $\mu I\le C^2\le I$ for some scalars
$0<m<M$ and $\mu>0$. Then for every $r>1$,
$$
CA^rC\le K(h_1,r)(CAC)^r,
$$
where $h_1:=M/m\mu$ and $K(h,r)$ is defined by \eqref{F-5.1}.
\end{lemma}

\begin{proof}
Since $r>1$, note that $t^r$ is convex on $t>0$ and so $t^r\le\alpha t+\beta$ for all
$t\in[m\mu,M]$, where
$$
\alpha:={M^r-(m\mu)^r\over M-m\mu}\qquad\mbox{and}\qquad
\beta:={M(m\mu)^r-m\mu M^r\over M-m\mu}.
$$
Since $m\mu I\le CAC\le MI$ and $\mu I\le C^2\le I$, we have
\begin{align*}
CA^rC&\le C(\alpha A+\beta I)C=\alpha CAC+\beta C^2 \\
&\le\alpha CAC+\beta I
\le\biggl(\max_{m\mu\le t\le M}{\alpha t+\beta\over t^r}\biggr)(CAC)^r \\
&=K(h_1,r)(CAC)^r.
\end{align*}
\end{proof}

Let $\cK$ be a Hilbert space and $\Phi:B(\cK)\to B(\cH)$ be a unital positive linear map.
Then it is known \cite[Theorem 3.18]{FMPS} that if $A\in\bP(\cK)$ with $mI\le A\le MI$ for
some $0<m<M$, then
$$
\Phi(A^r)\le K(M/m,r)\Phi(A)^r.
$$
Apply this to the linear map $\Phi:B(\oplus_1^n\cH)\to B(\cH)$,
$\Phi\bigl([A_{ij}]_{i,j=1}^n\bigr):=\sum_{j=1}^nw_jA_{jj}$, where $(w_1,\dots,w_n)$ is a
probability vector and $[A_{ij}]_{i,j=1}^n\in B(\oplus_1^n\cH)$ represented as an $n\times n$
matrix form with $A_{ij}\in B(\cH)$. Then we see that if $r>1$ and $A_1,\dots,A_n\in\bP$ with
$mI\le A_j\le MI$ for all $j$, then
\begin{align}\label{F-5.3}
\sum_{j=1}^nw_jA_j^r\le K(M/m,r)\Biggl(\sum_{j=1}^rw_jA_j\Biggr)^r.
\end{align}

In the rest of the section let $\omega=(w_1,\dots,w_n)$ be any probability measure.
The inequalities in the next theorem for the $n$-variable power means $P_{\omega,\alpha}$
are the reverse counterparts of \eqref{F-3.9} and \eqref{F-3.10}.

\begin{thm}\label{T-5.2}
Let $A_1,\dots,A_n\in\bP$ be such that $mI\le A_j\le MI$ for all $j$ for some scalars $0<m<M$.
Let $\kappa_0:=M/m$ and $\kappa(X):=\|X\|_\infty/\lambda_{\min}(X)$, the condition number of
$X:=P_{\omega,\alpha}(A_1,\dots,A_n)$. Then for any $\alpha\in(0,1]$ and $r\ge1$,
\begin{align}
&P_{\omega,\alpha}(A_1^r,\dots,A_n^r) \nonumber\\
&\le K(\kappa_0\kappa(X),r)K((\kappa_0\kappa(X))^\alpha,r)^{1/\alpha}
\lambda_{\min}^{r-1}(P_{\omega,\alpha}(A_1,\dots,A_n))P_{\omega,\alpha}(A_1,\dots,A_n),
\label{F-5.4}
\end{align}
and for any $\alpha\in[-1,0)$ and $r\ge1$,
\begin{align}
&P_{\omega,\alpha}(A_1^r,\dots,A_n^r) \nonumber\\
&\ge K(\kappa_0\kappa(X),r)^{-1}K((\kappa_0\kappa(X))^{-\alpha},r)^{1/\alpha}
\|P_{\omega,\alpha}(A_1,\dots,A_n)\|_\infty^{r-1}P_{\omega,\alpha}(A_1,\dots,A_n).
\label{F-5.5}
\end{align}
\end{thm}

\begin{proof}
Recall that $(P_{\omega,\alpha})^*=P_{\omega,-\alpha}$, $M^{-1}I\le A_j^{-1}\le m^{-1}I$ and
$\kappa(X^{-1})=\kappa(X)$. Hence \eqref{F-5.5} follows from \eqref{F-5.4} by replacing $A_j$
with $A_j^{-1}$ and $\alpha$ with $-\alpha$. So we may prove \eqref{F-5.4} only. Let
$0<\alpha\le1$, $r\ge1$, $X:=P_{\omega,\alpha}(A_1,\dots,A_n)$, and
$\lambda:=\lambda_{\min}(X)$. Since $X=\sum_{j=1}^nw_j(X\#_\alpha A_j)$, we have
\begin{align}\label{F-5.6}
I=\sum_{j=1}^nw_j(X^{-1/2}A_jX^{-1/2})^\alpha.
\end{align}
Since $\kappa(X)^{-1}I=\lambda\|X\|_\infty^{-1}I\le\lambda X^{-1}\le I$, it follows from
Lemma \ref{L-5.1} that
$$
(\lambda^{1/2}X^{-1/2})A_j^r(\lambda^{1/2}X^{-1/2})
\le K(\kappa_0\kappa(X),r)\bigl[(\lambda^{1/2}X^{-1/2})A_j(\lambda^{1/2}X^{-1/2})\bigr]^r.
$$
Hence letting $K_1:=K(\kappa_0\kappa(X),r)$ one has
\begin{align}\label{F-5.7}
\lambda^{1-r}K_1^{-1}(X^{-1/2}A_j^rX^{-1/2})\le(X^{-1/2}A_jX^{-1/2})^r
\end{align}
so that
$$
\sum_{j=1}^nw_j\bigl[\lambda^{1-r}K_1^{-1}(X^{-1/2}A_j^rX^{-1/2})\bigr]^\alpha
\le\sum_{j=1}^nw_j(X^{-1/2}A_jX^{-1/2})^{\alpha r}.
$$
Furthermore, since $(m/\|X\|_\infty)I\le X^{-1/2}A_jX^{-1/2}\le(M/\lambda)I$, one has
$$
(m/\|X\|_\infty)^\alpha I\le(X^{-1/2}A_jX^{-1/2})^\alpha\le(M/\lambda)^\alpha I.
$$
Hence by \eqref{F-5.3} and \eqref{F-5.6},
\begin{align*}
&\sum_{j=1}^nw_j\bigl[\lambda^{1-r}K_1^{-1}(X^{-1/2}A_j^rX^{-1/2})\bigr]^\alpha \\
&\qquad\le K((\kappa_0\kappa(X))^\alpha,r) 
\Biggl[\sum_{j=1}^nw_j(X^{-1/2}A_jX^{-1/2})^\alpha\Biggr]^r. 
\end{align*}
Letting $K_2:=K((\kappa_0\kappa(X))^\alpha,r)$ we obtain
$$
\sum_{j=1}^nw_j\bigl[\lambda^{1-r}K_1^{-1}K_2^{-1/\alpha}(X^{-1/2}A_jX^{-1/2})
\bigr]^\alpha\le I,
$$
or equivalently,
$$
\sum_{j=1}^nw_j\bigl[(\lambda^{r-1}K_1K_2^{1/\alpha}X)\#_\alpha A_j^r\bigr]
\le\lambda^{r-1}K_1K_2^{1/\alpha}X.
$$
By Theorem \ref{T-2.1}\,(3) this implies that
$$
P_{\omega,\alpha}(A_1^r,\dots,A_n^r)
\le\lambda^{r-1}K_1K_2^{1/\alpha}P_{\omega,\alpha}(A_1,\dots,A_n),
$$
which is \eqref{F-5.4}.
\end{proof}

Next, let $M$ be a general $n$-variable operator mean and $\sigma$ be any operator mean with
$\sigma\ne\frak{l}$. The following is the reverse counterpart of \eqref{F-4.1}.

\begin{thm}\label{T-5.3}
Let $A_1,\dots,A_n\in\bP$ be such that $mI\le A_j\le MI$ for all $j$ for some scalars $0<m<M$.
Let $\kappa_0:=M/m$ and $\kappa(X):=\|X\|_\infty/\lambda_{\min}(X)$, where
$X:=M_\sigma(A_1,\dots,A_n)$. Then for any $r\ge1$,
\begin{align}
&K(\kappa_0\kappa(X),r)^{-1}\|M_\sigma(A_1,\dots,A_n)\|_\infty^{r-1}
M_\sigma(A_1,\dots,A_n) \nonumber\\
&\qquad\le M_{\sigma_{1/r}}(A_1^r,\dots,A_n^r) \nonumber\\
&\qquad\le K(\kappa_0\kappa(X),r)\lambda_{\min}^{r-1}(M_\sigma(A_1,\dots,A_n))
M_\sigma(A_1,\dots,A_n). \label{F-5.8}
\end{align}
\end{thm}

\begin{proof}
As in the proof of Theorem \ref{T-4.1} we may prove the second inequality only. Let
$\lambda:=\lambda_{\min}(X)$ and $K_1:=K(\kappa_0\kappa(X),r)$. Then we have \eqref{F-5.7}
so that
$$
\bigl[\lambda^{1-r}K_1^{-1}(X^{-1/2}A_j^rX^{-1/2})\bigr]^{1/r}
\le X^{-1/2}A_jX^{-1/2}.
$$
Now the remaining proof is similar to that of Theorem \ref{T-4.1}, whose details may be
omitted.
\end{proof}

When $M=\cA_\omega$ or $\cH_\omega$ and $\sigma=\#_\alpha$, \eqref{F-5.8} gives the following:
For any $\alpha\in[-1,1]\setminus\{0\}$ and $r\ge1$,
\begin{align}
&K(\kappa_0\kappa(X),r)^{-1}\|P_{\omega,\alpha}(A_1,\dots,A_n)\|_\infty^{r-1}
P_{\omega,\alpha}(A_1,\dots,A_n) \nonumber\\
&\qquad\le P_{\omega,\alpha/r}(A_1^r,\dots,A_n^r) \nonumber\\
&\qquad\le K(\kappa_0\kappa(X),r)\lambda_{\min}^{r-1}(P_{\omega,\alpha}(A_1,\dots,A_n))
P_{\omega,\alpha}(A_1,\dots,A_n), \label{F-5.9}
\end{align}
where $X:=P_{\omega,\alpha}(A_1,\dots,A_n)$.

Letting $\alpha\to0$ in \eqref{F-5.9} gives the reverse counterpart of \eqref{F-3.13}.
This is also given by letting $\alpha\to0$ in \eqref{F-5.4} and \eqref{F-5.5} in view of
\eqref{F-5.2}.

\begin{cor}\label{C-5.4}
Let $A_j$, $\kappa_0$ and $\kappa(X)$ be as in Theorem \ref{T-5.2}, where
$X:=G_\omega(A_1,\dots,A_n)$. Then for any $r\ge1$,
\begin{align}
&K(\kappa_0\kappa(X),r)^{-1}\|G_\omega(A_1,\dots,A_n)\|_\infty^{r-1}
G_\omega(A_1,\dots,A_n) \nonumber\\
&\qquad\le G_\omega(A_1^r,\dots,A_n^r) \nonumber\\
&\qquad\le K(\kappa_0\kappa(X),r)\lambda_{\min}^{r-1}(G_\omega(A_1,\dots,A_n))
G_\omega(A_1,\dots,A_n). \label{F-5.10}
\end{align}
\end{cor}

\begin{remark}\rm
Inequality \eqref{F-5.10} improves \eqref{F-3.13} in some situations. For instance, when
$r=2$, the case $K(\kappa_0\kappa(X),2)\lambda_{\min}(X)<\|X\|_\infty$ for
$X=G_\omega(A_1,\dots,A_n)$ occurs, which is equivalent to
$$
(\kappa_0\kappa(X)+1)^2<4\kappa_0\kappa(X)^2.
$$
In fact, this happens when $1<\kappa_0<4$ and $\kappa(X)>[\sqrt{\kappa_0}(2-\sqrt{\kappa_0})]^{-1}$.
\end{remark}

\section{Optimality of $r\ge1$ or $r\le1$}

In previous sections we have shown different Ando-Hiai type inequalities involving the
power $r\ge1$ or $0<r\le1$. In this section we consider the problem of best possibility of
the condition on $r>0$ for some inequalities in Sections 3 and 4. To do this, we may confine
ourselves to $2$-variable operator means. For example, when the $n$-variable weighted power
mean $P_{\omega,\alpha}$ with a non-trivial weight $\omega$ is concerned, we can consider a
non-trivial $2$-variable power mean $(A,B)\mapsto P_{\omega,\alpha}(A,\dots,A,B)$. Then the
optimality problem for the $n$-variable case can be reduced to the $2$-variable case.

As for the weaker formulation in \eqref{F-3.5} restricted to the $2$-variable case, it was
shown in \cite[Corollary 3.1]{Wa2} that if $\sigma$ is a p.m.i.\ operator mean with
$\sigma\ne\frak{l},\frak{r}$ and $r>0$, then $A\sigma B\ge I$ $\implies$ $A^r\sigma B^r\ge I$
hods for every $A,B\in\bP$ if and only if $r\ge1$. We can directly verify this when $\sigma$
is the power mean treated in Corollary \ref{C-3.3}. Let $0<\alpha\le1$ and $0<w<1$. The weaker
formulation of \eqref{F-3.9} for scalars $A=a$ and $B=b$ implies that
$P_{w,\alpha}(a,b)\ge 1$ $\implies$ $P_{w,\alpha}(a^r,b^r)\ge 1$. For any $x>0$ let
$y:=P_{w,\alpha}(1,x)$; then we have $P_{w,\alpha}(1/y^r,x^r/y^r)\ge1$ and hence
$P_{w,\alpha}(1,x^r)\ge y^r$. This means that
$$
(1-w+wx^{\alpha r})^{1/\alpha}\ge(1-w+wx^\alpha)^{r/\alpha}\qquad\mbox{for $x>0$},
$$
which obviously holds only when $r\ge1$. Furthermore, the complementary version in
\eqref{F-3.11} for scalars $A=1$ and $B=x$ gives
$$
(1-w+wx^{\alpha r})^{1/\alpha}\le(1-w+wx^\alpha)^{r/\alpha}\qquad\mbox{for $x>0$}.
$$
This holds only when $r\le1$. From these arguments on the special case of power
means as well as the result in \cite{Wa2} mentioned above, we see that the condition $r\ge1$
or $r\le1$ is essential for Ando-Hiai type inequalities in Section 3. 

As for inequalities in Section 4 we focus on the inequalities in Corollary \ref{C-4.5}. Let
$\tau$ be an arbitrary operator mean with the representing function $f_\tau$. To consider
inequalities \eqref{F-4.8} and \eqref{F-4.9}, we define, for an arbitrary $r>0$,
$(f_\tau)_{[r]}(x):=f_\tau(x^r)^{1/r}$, $x>0$, and write
$$
A\tau_{[r]}B:=A^{-1/2}(f_\tau)_{[r]}(A^{-1/2}BA^{-1/2})A^{1/2},\qquad A,B\in\bP.
$$
Although $(f_\tau)_{[r]}$ for $r>1$ is not necessarily an operator monotone function so that
$\tau_{[r]}$ may not be an operator mean, the above expression $A\tau_{[r]}B$ indeed makes
sense even when $A\in\bP$ and $B\in B(\cH)^+$.

Concerning \eqref{F-4.9} we show the following:

\begin{prop}\label{P-6.1}
\begin{align}
(0,1]&=\bigl\{r>0:\|A\tau_{[r]}B\|_\infty^{r-1}(A\tau_{[r]}B)\le A^r\tau B^r,
\ A,B\in\bP\bigr\}, \label{F-6.1}\\
(0,1]&=\bigl\{r>0:A^r\tau B^r\le\lambda_{\min}^{r-1}(A\tau_{[r]}B)(A\tau_{[r]}B),
\ A,B\in\bP\bigr\}. \label{F-6.2}
\end{align}
\end{prop}

\begin{proof}
We may only prove that the inequality in \eqref{F-6.1} holds for all $A,B\in\bP$ only when
$r<1$, since the other direction is guaranteed by \eqref{F-4.9} and \eqref{F-6.2} follows
from \eqref{F-6.1} by replacing $\tau$ with $\tau^*$. So assume the inequality in \eqref{F-6.1}.
Let $A=\begin{bmatrix}1&0\\0&1\end{bmatrix}$ and $B=\begin{bmatrix}1&0\\0&x\end{bmatrix}$ with
$0<x\le1$. Then
$$
A\tau_{[r]}B=\begin{bmatrix}1&0\\0&f_\tau(x^r)^{1/r}\end{bmatrix},\qquad
A^r\tau B^r=\begin{bmatrix}1&0\\0&f_\tau(x^r)\end{bmatrix}.
$$
Since $\|A\tau_{[r]}B\|_\infty=1$, we have $A\tau_{[r]}B\le A^r\tau B^r$ and hence
$f_\tau(x^r)^{1/r}\le f_\tau(x^r)$, which implies that $f_\tau(x^r)^{r-1}\ge1$ for all
$x\in(0,1]$. Hence $r\le1$ follows.
\end{proof}

Finally, we give the optimality result for the weaker formulation of \eqref{F-4.8}. When
$\tau=\#_\alpha$ with $0<\alpha<1$, the result reduces to \cite[Theorem 3.1]{Wa2}.

\begin{prop}\label{P-6.2}
Assume that $\tau\ne\frak{l},\frak{r}$. Then
\begin{align}
[1,\infty)&=\bigl\{r>0:A,B\in\bP,\ A\tau B\le I\,\implies\,A^r\tau_{[1/r]}B^r\le I\bigr\},
\label{F-6.3}\\
[1,\infty)&=\bigl\{r>0:A,B\in\bP,\ A\tau B\ge I\,\implies\,A^r\tau_{[1/r]}B^r\ge I\bigr\}. \label{F-6.4}
\end{align}
\end{prop}

\begin{proof}
Since \eqref{F-6.4} follows from \eqref{F-6.3} as in the proof of Proposition \ref{P-6.1},
we may only prove that the property in \eqref{F-6.3} holds only when $r\ge1$. Assume the
property in \eqref{F-6.3}. Then the same holds for any positive definite matrix $A$ and
positive semidefinite matrix $B$. In fact, assume that $A\tau B\le I$, and for each $\eps>0$
define $B_\eps:=B+\eps I$, $\beta_\eps:=\|A\tau B_\eps\|_\infty$,
$\widetilde A_\eps:=\beta_\eps^{-1}A$ and $\widetilde B_\eps:=\beta_\eps^{-1}B_\eps$. Then,
since $\widetilde A_\eps\tau\widetilde B_\eps\le I$, we have
$\widetilde A_\eps^r\tau_{[1/r]}\widetilde B_\eps^r\le I$ so that
$A^r\tau_{[1/r]}B_\eps^r\le\beta_\eps^rI$. Letting $\eps\searrow0$ in both sides gives
$A^r\tau_{[1/r]}B^r\le I$.

Consider $2\times2$ matrices
$$
A=\begin{bmatrix}x^{-1}&0\\0&y^{-1}\end{bmatrix},\qquad x,y>0,
$$
\begin{align*}
B&=\begin{bmatrix}\sqrt t&-\sqrt{1-t}\\\sqrt{1-t}&\sqrt t\end{bmatrix}
\begin{bmatrix}1&0\\0&0\end{bmatrix}
\begin{bmatrix}\sqrt t&\sqrt{1-t}\\-\sqrt{1-t}&\sqrt t\end{bmatrix} \\
&=\begin{bmatrix}t&\sqrt{t(1-t)}\\\sqrt{t(1-t)}&1-t\end{bmatrix},
\qquad0<t<1.
\end{align*}
Compute
$$
A^{-1/2}BA^{-1/2}=\begin{bmatrix}xt&\sqrt{xyt(1-t)}\\\sqrt{xyt(1-t)}&y(1-t)\end{bmatrix}
=U\begin{bmatrix}xt+y(1-t)&0\\0&0\end{bmatrix}U^*
$$
with a unitary matrix
$$
U:=\begin{bmatrix}\sqrt{xt\over xt+y(1-t)}&-\sqrt{y(1-t)\over xt+y(1-t)}\\
\sqrt{y(1-t)\over xt+y(1-t)}&\sqrt{xt\over xt+y(1-t)}\end{bmatrix}.
$$
Hence we find that
\begin{align}
A\tau B\le I&\iff f_\tau(A^{-1/2}BA^{-1/2})\le A^{-1} \nonumber\\
&\iff\begin{bmatrix}f_\tau(xt+y(1-t))&0\\0&f_\tau(0)\end{bmatrix}\le
U^*\begin{bmatrix}x&0\\0&y\end{bmatrix}U \nonumber\\
&\iff(xt+y(1-t))\begin{bmatrix}f_\tau(xt+y(1-t))&0\\0&f_\tau(0)\end{bmatrix} \nonumber\\
&\qquad\qquad\le\begin{bmatrix}x^2t+y^2(1-t)&-(x-y)\sqrt{xyt(1-t)}\\
-(x-y)\sqrt{xyt(1-t)}&xy\end{bmatrix}. \label{F-6.5}
\end{align}
We divide the rest of the proof into the case $f_\tau(0)=0$ and the case $f_\tau(0)>0$.

{\it Case $f_\tau(0)=0$}.\enspace
Let $t=1/2$. Then
$$
A\tau B\le I\iff
\begin{bmatrix}{x+y\over2}f_\tau\bigl({x+y\over2}\bigr)&0\\0&0\end{bmatrix}
\le\begin{bmatrix}{x^2+y^2\over2}&-(x-y){\sqrt{xy}\over2}\\
-(x-y){\sqrt{xy}\over2}&xy\end{bmatrix},
$$
which is equivalent to
$$
\biggl\{{x^2+y^2\over2}-{x+y\over2}f_\tau\biggl({x+y\over2}\biggr)\biggr\}xy
\ge(x-y)^2{xy\over4}.
$$
That is,
$$
{2\over x+y}\,f_\tau\biggl({x+y\over2}\biggr)\le1.
$$
or equivalently,
\begin{align}\label{F-6.6}
f_{\tau'}\biggl({2\over x+y}\biggr)\le1,
\end{align}
where $\tau'$ is the transpose of $\tau$ so that $f_{\tau'}(x)=xf_\tau(1/x)$, $x>0$.
For any $r>0$, since $B^r=B$, we find by replacing $x,y,f_\tau$ with
$x^r,y^r,(f_\tau)_{[1/r]}$ that
\begin{align}
A^r\tau_{[1/r]}B^r\le I
&\iff{2\over x^r+y^r}\,(f_\tau)_{[1/r]}\biggl({x^r+y^r\over2}\biggr)\le1 \nonumber\\
&\iff\biggl({2\over x^r+y^r}\biggr)^{1/r}
f_\tau\biggl(\biggl({x^r+y^r\over2}\biggr)^{1/r}\biggr)\le1 \nonumber\\
&\iff f_{\tau'}\biggl(\biggl({2\over x^r+y^r}\biggr)^{1/r}\biggr)\le1. \label{F-6.7}
\end{align}
For any $0<r<1$ there are $x,y>0$ such that $\bigl({x^r+y^r\over2}\bigr)^{1/r}<1<{x+y\over2}$,
i.e., $\bigl({2\over x^r+y^r}\bigr)^{1/r}>1>{2\over x+y}$. Then \eqref{F-6.6} holds but
\eqref{F-6.7} does not since $f_{\tau'}$ is strictly increasing due to $\tau\ne\frak{r}$.

{\it Case $f_\tau(0)>0$}.\enspace
Let $r>0$. Similarly to \eqref{F-6.5} we find that
\begin{align}
&A^r\tau_{[1/r]}B^r\le I \nonumber\\
&\quad\iff(x^rt+y^r(1-t))
\begin{bmatrix}(f_\tau)_{[1/r]}(x^rt+y^r(1-t))&0\\0&(f_\tau)_{[1/r]}(0)\end{bmatrix}
\nonumber\\
&\qquad\qquad\le\begin{bmatrix}x^{2r}t+y^{2r}(1-t)&-(x^r-y^r)\sqrt{x^ry^rt(1-t)}\\
-(x^r-y^r)\sqrt{x^ry^rt(1-t)}&x^ry^r\end{bmatrix}. \label{F-6.8}
\end{align}
Set $\alpha:=f_\tau(0)$ and $\beta:=f_\tau'(1)$, so that $\alpha,\beta\in(0,1)$ since
$\tau\ne\frak{l},\frak{r}$. Note that $(f_\tau)_{[1/r]}(0)=\alpha^r$ and
$(f_\tau)_{[1/r]}'(1)=\beta$. Let $x=1+\eps$, $y=1-\eps$ and $t=(1+k\eps)/2$ for $\eps>0$ small
and some fixed $k\in\bR$ (that will be determined later). As $\eps\searrow0$, one can
approximately compute
$$
x^rt+y^r(1-t)=1+r\biggl(k+{r-1\over2}\biggr)\eps^2+o(\eps^2),
$$
$$
(f_\tau)_{[1/r]}(x^rt+y^r(1-t))=1+\beta r\biggl(k+{r-1\over2}\biggr)\eps^2+o(\eps^2),
$$
$$
x^{2r}t+y^{2r}(1-t)=1+r(2k+2r-1)\eps^2+o(\eps^2),
$$
$$
x^ry^r=1-r\eps^2+o(\eps^2),
$$
$$
(x^r-y^r)\sqrt{x^ry^rt(1-t)}=r\eps\sqrt{(1-r\eps^2)(1-k^2\eps^2)}+o(\eps^2).
$$
From these one can further compute the determinant of the difference of both sides of
\eqref{F-6.8} as follows:
\begin{align*}
&\Bigl\{(x^{2r}t+y^{2r}(1-t))-(x^rt+y^r(1-t))(f_\tau)_{[1/r]}(x^rt+y^r(1-t))\Bigr\} \\
&\times\Bigl\{x^ry^r-(x^rt+y^r(1-t))(f_\tau)_{[1/r]}(0)\Bigr\}
-\Bigl\{(x^r-y^r)\sqrt{x^ry^rt(1-t)}\Bigr\}^2 \\
&=\biggl\{1+r(2k+2r-1)\eps^2-\biggl[1+r\biggl(k+{r-1\over2}\biggr)\eps^2\biggr]
\biggl[1+\beta r\biggl(k+{r-1\over2}\biggr)\eps^2\biggr]+o(\eps^2)\biggr\} \\
&\quad\times\biggl\{1-r\eps^2-\alpha^r\biggl[
1+r\biggl(k+{r-1\over2}\biggr)\eps^2\biggr]+o(\eps^2)\biggr\}
-r^2\eps^2(1-r\eps^2)(1-k^2\eps^2)+o(\eps^2) \\
&=\biggl\{r\eps^2\biggl[(1-\beta)k+(2r-1)-(1+\beta){r-1\over2}\biggr]+o(\eps^2)\biggr\} \\
&\quad\times\biggl\{1-\alpha^r-r\eps^2\biggl[1+\alpha^r\biggl(k+{r-1\over2}\biggr)\biggr]
+o(\eps^2)\biggr\}
-r^2\eps^2+o(\eps^2) \\
&=r\eps^2\biggl\{(1-\alpha^r)\biggl[(1-\beta)k+(2r-1)-(1+\beta){r-1\over2}\biggr]-r\biggr\}
+o(\eps^2).
\end{align*}
Define
$$
F(r):={r\over1-\alpha^r}-(2r-1)+(1+\beta){r-1\over2},\qquad r>0,
$$
and compute
$$
F'(r)={1-\alpha^r+r\alpha^r\log\alpha\over(1-\alpha^r)^2}-2+{1+\beta\over2}
<{1-\alpha^r+\alpha^r\log\alpha^r\over(1-\alpha^r)^2}-1
$$
thanks to $0<\beta<1$. Letting $v:=\alpha^r\in(0,1)$ for $r>0$ one has
$$
1-v+v\log v-(1-v)^2=v(\log v+1-v)<0.
$$
Hence $F'(r)<0$ for all $r>0$ so that $F(r)>F(1)$ for all $r\in(0,1)$. Therefore, for any
$r\in(0,1)$ one can choose a $k$ so that $(1-\beta)k>F(1)$ but $(1-\beta)k<F(r)$. Then,
for $x=1+\eps$, $y=1-\eps$ and $t=(1+k\eps)/2$ with small $\eps>0$, it follows that
$A\tau B\le I$ holds but $A^r\tau_{[1/r]}B^r\le I$ does not, as desired.
\end{proof}

\section{Concluding remarks}

In the last section some remarks are in order.

\medskip
{\bf 1.}\enspace
A main tool in the present paper is Theorem \ref{T-2.1} on the deformed mean. In particular,
part (3) of Theorem \ref{T-2.1} has played a key role to prove Ando-Hiai type inequalities
in Sections 3--5. Although we confine ourselves to showing Theorem \ref{T-2.1} in the
$n$-variable setting, a more comprehensive treatment of the deformed mean $M_\sigma$ has been
developed in \cite{HLL} in a more general setting of (bounded) probability measures on $\bP$.
Then, various inequalities including Ando-Hiai's inequality can be proved in a similar way
for means (or barycenters) of probability measures on $\bP$.

\medskip
{\bf 2.}\enspace
The operator means of Kubo-Ando \cite{KA} are binary operations on $B(\cH)^+$ by
definition. However, studies of $n$-variable operator means so far have mostly been developed
for means on $\bP^n$, the $n$ positive invertible operators. Our study here is also
concentrated on means on $\bP^n$, but it is easy to extend an operator mean on $\bP^n$
satisfying conditions (I)--(IV) in Section 2 to $B(\cH)^+$ in the usual way that
\begin{align}\label{F-7.1}
M(A_1,\dots,A_n):=\lim_{\eps\searrow0}M(A_1+\eps I,\dots,A_n+\eps I)
\quad\mbox{in SOT},\quad A_j\in B(\cH)^+.
\end{align}
Then one can easily extend some inequalities in the paper to operators in $B(\cH)^+$, as done
in (2) of Corollaries \ref{C-4.4} and \ref{C-4.5} in the $2$-variable situation. The extension
of $M$ to $(B(\cH)^+)^n$ by \eqref{F-7.1} has recently been studied in \cite{FS} for the
$n$-variable Karcher and the power means and more general $n$-variable operator means.

\medskip
{\bf 3.}\enspace
Ando-Hiai's inequality was originally used in \cite{AH} to establish the log-majorization for
the geometric mean by using the antisymmetric tensor power technique. The latter technique can
work in the case of the Karcher mean as well, as shown in \cite{BK}. By this and Ando-Hiai's
inequality in \eqref{F-3.13} (first proved in \cite{Ya1}) the log-majorization in \cite{AH}
can extend to the Karcher mean, as noted in \cite[Remark 3.8]{HiPe} (see also \cite{HL}),
which has been referred to in Remark \ref{R-3.8}. It is worth noting that the complementary
inequality in \eqref{F-3.14} can be used in a similar way to show a different log-majorization
for the Karcher mean. Indeed, for every $N\times N$ positive definite matrices $A_1,\dots,A_n$
and any $r\in(0,1]$ we have
$$
\bigl(\lambda_i(G_\omega(A_1^r,\dots,A_n^r))\bigr)_{i=1}^N
\prec_{\log}\bigl(\lambda_{N+1-i}^{r-1}(G_\omega(A_1,\dots,A_n)
\lambda_i(G_\omega(A_1,\dots,A_n))\bigr)_{i=1}^N,
$$
that is, for $1\le k\le N$,
$$
\prod_{i=1}^k\lambda_i(G_\omega(A_1^r,\dots,A_n^r))
\le\prod_{i=1}^k\lambda_{N+1-i}^{r-1}(G_\omega(A_1,\dots,A_n)
\lambda_i(G_\omega(A_1,\dots,A_n))
$$
with equality for $k=N$, where $(\lambda_i(X))_{i=1}^N$ denotes the eigenvalues of an
$N\times N$ positive definite matrix $X$ in decreasing order counting multiplicities (see
\cite{AH} for more details on log-majorization).

\section*{Acknowledgements}

The work of F.~Hiai and Y.~Seo was supported in part by Grant-in-Aid for Scientific
Research (C)17K05266 and (C)16K05253, respectively.

\end{document}